\documentclass[12pt]{article}

\usepackage{fullpage}

\usepackage{amsfonts,amsthm}
\usepackage{graphicx}
\usepackage{epstopdf}
\usepackage{algorithmic}
\usepackage{tikz}
\usetikzlibrary{shapes.geometric}
\usetikzlibrary{shapes.misc, positioning}
\usepackage{amsmath,amssymb, amsfonts,  amstext, setspace, color, enumitem, hyperref, cite, tabularx,tabulary, subfig}
\usepackage{amsbsy,mathrsfs,eucal}
\usepackage[symbol]{footmisc}

\ifpdf
\DeclareGraphicsExtensions{.eps,.pdf,.png,.jpg}
\else
\DeclareGraphicsExtensions{.eps}
\fi

\newtheorem{theorem}{Theorem}[section]

\newtheorem{lemma}{Lemma}[section]
\theoremstyle{remark}
\newtheorem{remark}{Remark}[section]

\title{A \dong{new} conservative discontinuous Galerkin method via implicit penalization for the generalized KdV equation
}

\author{Yanlai Chen\thanks{Department of
		Mathematics, University of Massachusetts Dartmouth, North Dartmouth, MA 02747, USA. The work of this author was partially supported by a grant from the College of Arts \& Sciences at the University of Massachusetts Dartmouth, and by the UMass Dartmouth Marine and UnderSea Technology (MUST) Research Program made possible via an Office of Naval Research grant N00014-20-1-2849. Email: {\tt yanlai.chen@umassd.edu}.
	    }
	\and Bo Dong\thanks{Department of
		Mathematics, University of Massachusetts Dartmouth, North Dartmouth, MA 02747, USA. The work of this author was supported by the National Science Foundation (grant DMS-1818998). Email: {\tt bdong@umassd.edu}.}
	\and Rebecca\thanks{Department of
		Mathematics, University of Massachusetts Dartmouth, North Dartmouth, MA 02747, USA. The work of this author was supported by the National Science Foundation (grant DMS-1818998). Email: {\tt rpereira2@umassd.edu}.}}

\usepackage{amsopn}

\theoremstyle{definition}

\theoremstyle{remark}

\numberwithin{equation}{section}

\usepackage{multirow}
\usepackage{xcolor}
\definecolor{blck}{rgb}{0,0,0}

\definecolor{darkred}{rgb}{.6,.1,0}

\definecolor{blue}{rgb}{0,0,1}

\definecolor{red}{rgb}{1,0,0}

\newcommand\bo[1]{\textcolor{blck}{#1}}
\newcommand\dong[1]{\textcolor{blck}{#1}}

\newcommand{\jmp}[1]{[\![#1]\!]}                     

\newcommand{\Eh}{\mathscr{E}_h}

\def\calT{{\mathcal T}}

\theoremstyle{definition}

\numberwithin{equation}{section}

\def\calT{{\mathcal T}}

\newcommand{\eps}{\varepsilon}

\newlist{todolist}{itemize}{2}
\setlist[todolist]{label=$\square$}
\usepackage{pifont}
\begin{document}

\date{}


\maketitle

\begin{abstract}
	We design, analyze, and implement a \dong{new} conservative Discontinuous Galerkin (DG) method for the simulation of  solitary wave solutions to the generalized Korteweg-de Vries (KdV) Equation.  The key feature of our method is the conservation, at the numerical level, of the mass, energy and Hamiltonian that are conserved by exact solutions of all KdV equations. To our knowledge, this is the first DG method that conserves all these three quantities, a property critical for the accurate long-time evolution of solitary waves. To achieve the desired conservation properties, our novel idea is to introduce two stabilization parameters in the numerical fluxes as new unknowns which then allow us to enforce the conservation of energy and Hamiltonian in the formulation of the numerical scheme. We prove the conservation properties of the scheme which are corroborated by numerical tests. This idea of achieving conservation properties by implicitly defining penalization parameters, that are traditionally specified {\em a priori}, can serve as a framework for designing physics-preserving
	numerical methods for other types of problems.
\end{abstract}



\section{Introduction}
In this paper, we consider the following generalized Korteweg-de Vries (KdV) equation
\begin{equation}\label{eq:Kdv_Equation}
    u_{t} +\varepsilon u_{xxx} + f(u)_{x} = g(x,t), \qquad x\in\Omega=[a, b],\, t>0
\end{equation}
 with periodic boundary conditions 
 and the initial condition $u(x,0)=u_{0}(x)$. Here, $f(u)$ is usually some polynomial of $u$. When \bo{$\varepsilon=1$,} $f(u)=3u^2$ and $g \equiv 0$, \eqref{eq:Kdv_Equation} represents the \bo{original} KdV equation.
 
KdV equations are widely adopted to model one-dimensional long waves and have applications in plasma physics, biology, nonlinear optics, quantum mechanics, and fluid mechanics; \bo{see \cite{HammackSegur74,HammackSegur78,Osborne90, GardnerMorikawa60,vanWijngaarden72,Kluwick83,HelfrichWhitehead89}}. There is also a lot of interest in theoretical studies on the mathematical properties of solutions to KdV equations. Many modern areas of mathematics and theoretical physics opened up thanks to the basic research into the KdV equations. As a consequence, there have been intense efforts on developing 
numerical methods for KdV equations, including  finite difference methods  \cite{Vliegenthart71,Goda77,LiVisbal06}, finite element methods  \cite{Winther80,ArnoldWinther82,SANZSERNA1981,BakerDougalisKarakashian83}, spectral
methods   \cite{FornbergWhitham78,HuangSloan92,GuoShen01,MaSun01} and operator slitting methods \cite{Holden1999,Tao11}.

KdV equations feature a combination of the nonlinear term and the dispersive term $u_{xxx}$, which makes it difficult to  achieve numerical properties such as stability and convergence. Moreover, it is known that KdV equations may have ``blow-up" solutions but the mechanism of the singularity formation is not clear  \cite{Merle01,MartelMerle02}. The study in \cite{BonaDougalisKarakashianMcKinney90} showed that the simulation of blow-up solutions, almost for sure, will require highly nonuniform meshes. This makes Discontinuous Galerkin (DG) methods suitable for solving KdV equations  due to their advantages including high-order accuracy, compact stencil, capability of handling nonuniform meshes and variable degrees,  and flexibility in constructing the numerical fluxes to achieve conservation of particular physical quantities. 
DG methods  \cite{shu2009discontinuous, YanShu02, xu2005local, ChengShu08, XuShu12, HuffordXing14, chen2016new} have been developed for KdV type equations.
In particular, there have been continuous efforts on developing DG methods that conserve physically  interesting quantities of their solutions. Indeed, all KdV equations have three such quantities:
$$ \textrm{Mass: } \int_\Omega u dx, \qquad \textrm {Energy:} \int_\Omega u^2 dx, \qquad \textrm {Hamiltonian:} \int_\Omega (\frac{\varepsilon}{2}u_x^2-V(u)) dx,$$where $V(\cdot)$ is an anti-derivative of $f(\cdot)$. 
This property is \bo{crucial} for their solitary wave solutions to maintain amplitude, shape, and speed even after colliding with another such wave. 
Numerical results \cite{bona2013conservative,KarakashianXing16,liu2016hamiltonian,zhang2019conservative}  showed that DG methods preserving these invariants  can maintain numerical stability over a long time period and help reduce phase and shape error after long time integration. However, existing conservative DG methods cannot conserve the energy and Hamiltonian simultaneously though the conservation of mass is easy to achieve. In Table \ref{tab:WhoConserve} we list some conservative DG methods for KdV equations. This is in no way an exhaustive list, but it shows the trend and main efforts in the development of conservative DG methods for KdV equations. 
We can see that the methods in \cite{bona2013conservative,yi2013direct, KarakashianXing16,chen2016new} and the first method in \cite{zhang2019conservative} conserve the energy but not the Hamiltonian, while the method in \cite{liu2016hamiltonian} and the second method in \cite{zhang2019conservative} conserve the Hamiltonian but not the energy.
\begin{table}[h!]
\centering
     \begin{tabular}{|p{6.5cm}|p{1cm}|p{2.0cm}|p{1.5cm}|} 
    \hline
\multicolumn{1}{|c|}{Method} & Year & Hamiltonian & Energy  \\
    \hline 
{Conservative DG for the Generalized KdV (GKdV) \cite{bona2013conservative}}  &  \multicolumn{1}{c|}{\multirow{2}{*}{2013}} &  \multicolumn{1}{c|}{\multirow{2}{*}{\ding{55}}}  &   \multicolumn{1}{c|}{\multirow{2}{*}{\ding{51}}} \\
    \hline
{Direct DG for GKdV \cite{yi2013direct}} & \multicolumn{1}{c|}{\multirow{1}{*}{2013}}  & \multicolumn{1}{c|}{\multirow{1}{*}{\ding{55}}}   &   \multicolumn{1}{c|}{\multirow{1}{*}{\ding{51}}} \\
   \hline
{Conservative LDG for GKdV \cite{KarakashianXing16}}
   & \multicolumn{1}{c|}{\multirow{1}{*}{2016}}   & \multicolumn{1}{c|}{\multirow{1}{*}{\ding{55}}}   &   \multicolumn{1}{c|}{\multirow{1}{*}{\ding{51}}}  \\
\hline 
{$H^2$-Conservative DG for Third-Order Equations\cite{chen2016new}} & \multicolumn{1}{c|}{\multirow{2}{*}{2016}}
   &\multicolumn{1}{c|}{\multirow{2}{*}{\ding{55}}}   &   \multicolumn{1}{c|}{\multirow{2}{*}{\ding{51}}}\\
   \hline
{Hamiltonian-Preserving DG for GKdV \cite{liu2016hamiltonian}} & \multicolumn{1}{c|}{\multirow{1}{*}{2016}} &  \multicolumn{1}{c|}{\multirow{1}{*}{\ding{51}}}   &   \multicolumn{1}{c|}{\multirow{1}{*}{\ding{55}}}  \\
    \hline
{Conservative and Dissipative LDG for KdV  \cite{zhang2019conservative}, \quad Scheme I} &
   \multicolumn{1}{c|}{\multirow{2}{*}{2019}} &
   \multicolumn{1}{c|}{\multirow{2}{*}{\ding{51}}}   &   \multicolumn{1}{c|}{\multirow{2}{*}{\ding{55}}} \\
   \hline
{Conservative and Dissipative LDG for KdV \cite{zhang2019conservative}, \quad Scheme II} &
   \multicolumn{1}{c|}{\multirow{2}{*}{2019}} &
   \multicolumn{1}{c|}{\multirow{2}{*}{\ding{55}}}   &   \multicolumn{1}{c|}{\multirow{2}{*}{\ding{51}}} \\    
 \hline
    \end{tabular}\\
    \caption{\label{tab:WhoConserve} The conservation properties of the previous DG methods}
\end{table} 
Most of these conservative DG methods have an optimal convergence order for even degree polynomials and sub-optimal order for odd degree polynomials except that the Hamiltonian conserving method in \cite{zhang2019conservative} has optimal convergence order for any polynomial degrees.

In this work, we  develop a new DG method for KdV equations that conserves all three invariants: mass, energy, and Hamiltonian.
 This conservative DG method will allow us to model and simulate the soliton wave more accurately over a long time period. 
 Our novel idea on designing the method is to treat   the penalization/stabilization parameters in the numerical fluxes implicitly (i.e.\dong{,} as new unknowns), which allow  two more equations in the formulation of the DG method that explicitly enforce the conservation of energy and Hamiltonian. The stabilization parameters are solved together with the approximations of the exact solutions. 
  Due to the time-step constraint implied by the third-order spatial derivative, we use implicit time marching schemes to avoid extremely small time steps.
  Since our DG scheme for spatial discretization is conservative, in implementation we use 
  the implicit midpoint method which is conservative for time discretization. Our numerical results show that,  just like most other conservative DG methods in literature, our method has optimal convergence for the even polynomial degrees and sub-optimal convergence for the odd ones. More significantly, our method can conserve both the energy and the Hamiltonian over a long time period. 
  
\dong{
	As shown in Table \ref{tab:WhoConserve}, both standard DG and LDG methods have appeared in literature to achieve conservation.  
We choose the LDG-like framework for our method because it has three numerical traces, and thus more room for tuning for better conservation properties. We would like to point out that
	 our method has computational complexity that is only negligibly more than standard LDG. When the equation is nonlinear (which is our focus), both discretized systems are nonlinear thus needing iterative solvers. Standard LDG system has $3 N (k+1)$ equations when $N$ elements and polynomials of degree $k$ are used. Our system has $3 N (k+1) + 2$ equations due to the introduction of two new unknown (constant) parameters.} 
We would like to \dong{further} remark that our idea of enforcing conservation properties by using implicit stabilization parameters can be applied to develop new conservative methods for other types of problems that feature conservation of physical quantities. \dong{It can also be extended to preserve more invariants for the KdV equation by introducing more than two implicit stabilization parameters.} This opens the door to promising future extensions.

The rest of the paper is structured as follows: Section 2 will describe the formulation of our DG method and prove the conservation properties. Implementation of our method is briefly discussed in Section 3\bo{,} leaving further details \bo{to} the Appendix. We display numerical results on solving third-order linear and nonlinear equations and the classical KdV equation, showing the order of convergence and conservation properties we have observed in our numerical experiments in Section 4. Finally, we end with concluding remarks in Section 5.

\section{Main Results}
\label{sec:mainresults}
In this section, we discuss our main results. We start by introducing our notations. Next, we describe our DG method and discuss the choice of penalization parameters that ensure the conservation of the Hamiltonian and energy. After that, we prove that our numerical solutions do conserve the three invariants: mass, energy, and Hamiltonian. 

\subsection{Notation} To define our DG method, first let us introduce some notations.
We partition the domain $\Omega = (a,b)$ as 
\[
\calT_h=\{I_i:=(x_{i-1}, x_i): a=x_0<x_1<\cdots<x_{N-1}<x_N=b\}.
\]
We use $\partial\calT_h:=\{ \partial I_i: i=1,\dots,N\}$ to denote the set of all element boundaries, and $\Eh:=\{x_i\}_{i=0}^N$ to denote all the nodes.
We also set $h_i = x_i - x_{i-1}$ and $h:=\max_{1\le i\le N} h_i$.

For any function $\zeta\in L^2(\partial\calT_h)$, we denote
its values on $\partial I_i:=\{x^+_{i-1}, x^-_i\}$ by $\zeta(x_{i-1}^+)$ (or simply $\zeta^+_{i-1}$) and $\zeta(x_i^-)$ (or simply $\zeta^-_i$).
Note that $\zeta(x_{i}^+)$ does not have to be equal to $\zeta(x_i^-)$.
In contrast, for any function $\eta\in L^2(\Eh)$, its value at $x_i$, $\eta(x_i)$ (or simply $\eta_i$) is uniquely defined; in this case,
 $\eta(x_i^-) = \eta(x_{i}^+) = \eta(x_i)$.

We let
$$(\varphi, v):=\sum_{i=1}^{N} (\varphi,v)_{I_i}, \quad \langle \varphi, vn\rangle:=\sum_{i=1}^{N}\langle\varphi, vn\rangle_{\partial I_i},$$
where 
$$(\varphi, v)_{I_i}=\int_{I_i} \varphi v dx,  \quad \langle \varphi,v n\rangle_{\partial I_i}=\varphi(x_i^-)v(x_i^-)n(x_i^-) +\varphi(x_{i-1}^+)v(x_{i-1}^+)n(x_{i-1}^+).$$ 
Here $n$ denotes the outward unit normal to $I_i$, that is $n(x_{i-1}^+):=-1$ and $n(x_i^-):=1$.
We define the average and jump of $\varphi$ as 
$$\{\varphi\}(x_i):=\frac{1}{2}\big(\varphi(x_i^-)+\varphi(x_i^+)\big), 
\quad \jmp{\varphi}(x_i):=
\varphi(x_i^-)-\varphi(x_i^+).$$
We also define the finite element space
\begin{equation*}
{W}_h^k = \{\omega \in L^2(\mathcal{T}_h): \;\;  \omega|{{_K}} \in {P}_{k}(K) \textrm{ for any } K\in\mathcal{T}_h,\; \mbox{ and } \omega(a)=\omega(b)\},
\end{equation*}
where $P_k(D)$ is the space of piecewise polynomials of degree up to $k$ on the set $D$. Finally, the $H^s(D)$-norm is denoted by $\|\cdot\|_{s, D}$.
We drop the first subindex if $s=0$,  and the second if $D=\Omega$ or \dong{$\mathcal{T}_h$}.

\subsection{The DG method} 
\label{sec:dgm}
To define our DG method for the KdV equation \eqref{eq:Kdv_Equation},
 we first rewrite it as the following system of first-order equations
\begin{equation}\label{eq:kdvsystem}
	\begin{split}
q - u_x \,&=\, 0,   \qquad\;\,{\textrm {in} }\;  \Omega,\\
p -\eps  q_x\,&=\, f(u), \quad {\textrm {in} }\;  \Omega,\\
u_t+ p_x  \,&=\,g(x), \quad \, {\textrm {in} }\;  \Omega,
  \end{split}
\end{equation}
 with the initial condition $u(x,0)=u_0(x)$ and the periodic boundary conditions
$$u(a) = u(b), \quad q(a) = q(b),  \quad p(a) = p(b).$$

We discretize \eqref{eq:kdvsystem} by
seeking $(u_h, q_h, p_h)$ as approximations to $(u, q, p)$ in the space $\left(W_h^{k}\right)^3$  
such that 
\begin{subequations}
\label{eq:scheme_time}
\begin{alignat}{2}
\label{eq:scheme_time1}
({q}_h,{v}) + (u_h,v_x) - \langle \widehat{u}_h,{v}n \rangle  & = 0, & \forall v \in W_h^k,\\
\label{eq:scheme_time2}
({p}_h,{z})  + \eps (q_h,z_x)  - \eps \langle \widehat{q}_h,{z}n \rangle  & = (f(u_h), z), \quad & \forall z \in W_h^k,\\
\label{eq:scheme_time3}
({u}_{ht},{w})-(p_h,w_x) + \langle \widehat{p_h}, {w}n \rangle & = (g, w),& \forall w \in W_h^k.
\end{alignat}
\end{subequations}
Here, $\widehat{u_h}, \widehat{q_h}, \widehat{p_h}$ are the so-called numerical traces whose definitions are in general  critical for the accuracy and stability of the DG method \cite{ArnoldBrezziCockburnMarini02}. There are multiple ways of defining them. We adopt the one that is similar to the Local Discontinuous Galerkin (LDG) methods \cite{ArnoldBrezziCockburnMarini02} 
\begin{subequations}
\label{eq:scheme_hats}
\begin{alignat}{1}
\label{eq:uhat}
\widehat{u_h} = & \{u_h\},\\
\label{eq:qhat}
\widehat{q_h} = & \{q_h\} + \tau_{qu} \jmp{u_h},\\
\label{eq:phat}
\widehat{p_h} = & \{p_h\} + \tau_{pu} \jmp{u_h}. 
\end{alignat}
\end{subequations}
The key difference is that, instead of specifying the values of the penalty parameters $(\tau_{qu}, \tau_{pu})$ as done by LDG \cite{ArnoldBrezziCockburnMarini02}, we leave them as unknowns. It is exactly due to the resulting freedom of placing two more constraints that, 
as shown in Lemma \ref{lemma:conserv_general} in the next section, the scheme is able to conserve the mass, $L^2$-energy, and the Hamiltonian of the numerical solutions.
Toward that end, we require that the penalization parameters $\tau_{qu}$ and $\tau_{pu}$ be constants that satisfy
\begin{subequations}
	\label{eq:trace1}
	\begin{alignat}{1}
	\label{eq:tau_pu1}
	&\tau_{pu}\sum_{i=1}^N \jmp{u_h}^2(x_i) - \eps\tau_{qu}\sum_{i=1}^N \jmp{u_h} \jmp{q_h}(x_i)=\sum_{i=1}^N \Big(\jmp{V(u_h)} - \{\varPi f(u_h)\} \jmp{u_h}\Big)(x_i),\\  
	\label{eq:tau_qu1}
	&\tau_{pu}\sum_{i=1}^N \jmp{p_h} \jmp{u_h}(x_i)+\eps\tau_{qu}\sum_{i=1}^N\jmp{u_h}_t\jmp{u_h}(x_i)=0.
	\end{alignat}
\end{subequations}
Here, $V(\cdot)$ is an antiderivative of $f(\cdot)$. In summary, our method is to seek $(u_h, q_h, p_h) \in \left(W_h^{k}\right)^3$  and penalty parameters $(\tau_{qu}, \tau_{pu})$ such that \eqref{eq:scheme_time1} - \eqref{eq:scheme_time3}\bo{,}  \eqref{eq:tau_pu1}\bo{,} and \eqref{eq:tau_qu1} are satisfied.

\begin{remark}
	Here we would like to point out that our scheme is not an LDG method. To our knowledge, existing LDG methods do not conserve the energy of solutions to KdV equations. The penalty parameters in LDG methods are known constants, while in our schemes $\tau_{qu}$ and \dong{$\tau_{pu}$} are considered as new unknowns. Correspondingly we have two more equations from \eqref{eq:trace1}. 
	In fact, we can write $\tau_{qu}$ and \dong{$\tau_{pu}$} in terms of $u_h, q_h,p_h$ as
	\begin{alignat*}{1}
	\tau_{qu}&=-\frac{1}{\eps}\frac{\eta(p_h, u_h)\sum_{i=1}^{N}\Big(\jmp{V(u_h)}-\{\Pi f(u_h)\}\jmp{u_h}\Big) }{\eta(q_h, u_h)\eta(p_h, u_h)+\eta(u_{ht}, u_h)\eta(u_h, u_h)},\\
	\tau_{pu}&= -\eps\frac{\eta(u_{ht}, u_h)}{\eta(p_h,u_h)} \tau_{qu}\dong{,}
	\end{alignat*}
	where we have used the notation $\eta(w,v)=\sum_{i=1}^N\jmp{w} \jmp{v}(x_i).$
	These expressions show that our method is different from LDG Methods. 
\end{remark}

\subsection{Conservative properties}
Now we discuss the conservation properties of the schemes in the previous section. First, in the following Lemma we give general conditions for $\widehat{u_h}, \widehat{q_h}, \widehat{p_h}$ under which DG methods that satisfy  \eqref{eq:scheme_time} conserve the mass, $L^2$ energy, and Hamiltonian. Then we apply the Lemma to prove the conservation properties for the DG method defined by \eqref{eq:scheme_time}-\eqref{eq:trace1}.  \begin{lemma}\label{lemma:conserv_general}
	Suppose $(u_h, q_h, p_h)$ satisfy \eqref{eq:scheme_time} with $g=0$. \\
	(i) If $\widehat{p_h}$ is single-valued, then we have
	\label{eq:mass_conserv}
	\[
	\frac{d}{dt}\int_{\calT_h} u_h \,dx=0,
	\quad \, {\rm (mass \,-\, conservation)}.
	\]
(ii)	If $\widehat{u_h}, \widehat{q_h}, \widehat{p_h}$ are single-valued and satisfy the condition
		\begin{alignat}{1}
		\label{eq:conserv_cond1}
0=&\sum_{i=1}^N  \Big(\jmp{V(u_h)} - \{\varPi f(u_h)\} \jmp{u_h} +(\jmp{\varPi f(u_h)} - \jmp{p_h}) (\widehat{u_h} - \{u_h\})\\
\nonumber
& \quad\quad-\jmp{u_h}(\widehat{p_h} - \{p_h\}) + \varepsilon\jmp{q_h} (\widehat{q_h} - \{q_h\}) \Big)(x_i),
\end{alignat}
then we have  
\begin{alignat}{1}
\label{eq:L2_conserv}
	\frac{d}{dt}\int_{\calT_h} u_h^2 \,dx=0,
& \quad \, ({\rm energy}-{\rm conservation}).
\end{alignat}
(iii)	If $\widehat{u_h}, \widehat{q_h}, \widehat{p_h}$ are single-valued and satisfy the condition
\begin{alignat}{1}
\label{eq:conserv_cond2}
0= & \sum_{i=1}^{N} \left( \jmp{p_h} (\widehat{p_h} - \{p_h\}) + \varepsilon\jmp{q_h} (\widehat{u_h} - \{u_h\})_t +\varepsilon \jmp{u_h}_t (\widehat{q_h} - \{q_h\})\right)(x_i),
		\end{alignat}
then we have
\begin{alignat}{1}
\label{eq:Hamiltonian_conserv}
	\frac{d}{dt}\int_{\calT_h} \Big(\,\frac{\varepsilon}{2}
 q_h^2 - V(u_h)\Big) \,dx=0,
& \quad\, ({\rm Hamiltonian}-{\rm conservation}).
\end{alignat}
\end{lemma}

\begin{proof}
	(i) To prove the mass conservation, we just need to take $w = 1$ in \eqref{eq:scheme_time3} and use the fact that $\widehat{p_h}$ is single-valued.
	
	(ii) Next, we prove the energy-conservation, which is also called $L^2$-conservation. We take $w := u_h$, $v := -p_h + \varPi f(u_h)$, and $z := q_h$ in \eqref{eq:scheme_time} and add the three equations together to get
	\begin{alignat*}{1}
	(f(u_h), q_h) = & (u_{ht}, u_h) - (p_h, u_{hx}) + \langle \widehat{p_h}, u_h n \rangle - (u_h, p_{hx})
	+ \langle \widehat{u}_h, p_hn \rangle + \varepsilon(q_h, q_{hx}) \\
	& - \varepsilon\langle \widehat{q}_h, q_h n\rangle + (q_h, \varPi f(u_h)) + \langle u_h - \widehat{u_h}, \varPi f(u_h) n\rangle - (\varPi f(u_h), u_{hx})
	\end{alignat*}
	Since $$(f(u_h), q_h) = (\varPi f(u_h), q_h)$$ and $$(\varPi f(u_h), u_{hx})=(f(u_h), u_{hx})= \langle V(u_h), n\rangle,$$ we have that
	\begin{alignat*}{1}
	0 = & (u_{ht}, u_h) - \langle p_h, u_h n\rangle + \langle \widehat{p_h}, u_h n\rangle + \langle \widehat{u}_h, p_h n\rangle
	+ \frac{\varepsilon}{2}\langle q_h^2,n \rangle -\varepsilon \langle \widehat{q}_h q_h,n\rangle\\
	& + \langle u_h - \widehat{u_h}, \varPi f(u_h) n\rangle -  \langle V(u_h), n\rangle\\
	= & (u_{ht}, u_h) - \langle \widehat{p_h} - p_h + \varPi f(u_h), (\widehat{u}_h - u_h)n\rangle
	+ \frac{\varepsilon }{2}\langle (q_h - \widehat{q_h})^2,n \rangle -  \langle V(u_h), n\rangle,
	\end{alignat*}
	where we have used the single-valuedness of numerical traces. This means that
	\begin{alignat*}{1}
	\frac{1}{2} \frac{d}{dt} (u_{h}, u_h) = &\langle V(u_h), n\rangle  + \langle \widehat{p_h} - p_h + \varPi f(u_h), (\widehat{u}_h - u_h)n\rangle
	-  \frac{\varepsilon}{2}\langle (q_h - \widehat{q_h})^2,n \rangle\\
	= &\sum_{i=1}^N  \left(\jmp{V(u_h)} - \{\varPi f(u_h)\} \jmp{u_h} +(\jmp{\varPi f(u_h)} - \jmp{p_h}) (\widehat{u_h} - \{u_h\}) \right.\\
	& \left.-\jmp{u_h}(\widehat{p_h} - \{p_h\}) + \varepsilon\jmp{q_h} (\widehat{q_h} - \{q_h\})\right)(x_i).
	\end{alignat*}
	Here, we used the equality $\langle \rho, v n \rangle = \sum_{i=1}^N (\jmp{\rho} \{ v\} + \jmp{v} \{ \rho\})(x_i)$ for any $\rho, v\in W_h^k$. 
	When the condition \eqref{eq:conserv_cond1} is satisfied, we immediately get the energy-conservation, \eqref{eq:L2_conserv}.

	(iii) To prove the Hamiltonian conservation properties in \eqref{eq:Hamiltonian_conserv}, we first differentiate the equation \eqref{eq:scheme_time1} with respect to $t$ to obtain
	\[
	({q}_{ht},{v})  + (u_{ht},v_x)  - \langle \widehat{u}_{ht},{v}n \rangle  = 0.
	\]
	Then, we take $v := \varepsilon q_{h}$ in the equation above, $z := u_{ht}$ in \eqref{eq:scheme_time2} and $w := -p_h$ in \eqref{eq:scheme_time3} and add the three equations together to get
	\begin{alignat*}{1}
	(f(u_h), u_{ht}) = \;& \varepsilon(q_{ht}, q_h) + (p_h, p_{hx}) - \langle \widehat{p_h}, p_h n \rangle \\
	&+ \varepsilon (u_{ht}, q_{hx}) + \varepsilon(q_h, u_{htx}) - \varepsilon\langle \widehat{u}_{ht}, q_hn \rangle
	-  \varepsilon \langle \widehat{q}_h, u_{ht} n\rangle.
		\end{alignat*}
	Since $\widehat{u_h}, \widehat{q_h}$, and $\widehat{p_h}$ are single-valued, we have	
	\begin{alignat*}{1}
		(f(u_h), u_{ht}) =\; & \varepsilon (q_{ht}, q_h) + \langle \frac{1}{2}  p_h^2,n\rangle - \langle \widehat{p_h} p_h, n \rangle + \varepsilon\langle u_{ht} - \widehat{u}_{ht}, (q_h - \widehat{q}_h)n\rangle\\
	= \;& \varepsilon(q_{ht}, q_h) + \frac{1}{2}\langle (p_h - \widehat{p_h})^2,n\rangle + \varepsilon\langle u_{ht} - \widehat{u}_{ht}, (q_h - \widehat{q}_h)n\rangle.
	\end{alignat*}
	This implies that
	\begin{alignat*}{1}
	&\frac{d}{dt} \left(\frac{\varepsilon}{2} (q_h, q_h) - (V(u_h),1) \right) \\
	= & \sum_{i=1}^N  \left( \jmp{p_h} (\widehat{p_h} - \{p_h\}) + \varepsilon\jmp{q_h} (\widehat{u_h} - \{u_h\})_t + \varepsilon\jmp{u_h}_t (\widehat{q_h} - \{q_h\})\right)(x_i).
	\end{alignat*}
 If the numerical traces satisfy \eqref{eq:conserv_cond2}, we get the conservation of the Hamiltonian \eqref{eq:Hamiltonian_conserv}. 	
	This concludes the proof of Lemma \ref{lemma:conserv_general}.
\end{proof}

Next we use Lemma \ref{lemma:conserv_general} to show that our scheme defined by \eqref{eq:scheme_time} - \eqref{eq:trace1} conserves the  mass, the $L^2$-energy, and the Hamiltonian of the numerical solutions.

\begin{theorem}
\label{thm:conserv}
For $(u_h, q_h, p_h)$ satisfying \eqref{eq:scheme_time} with $g=0$ and numerical traces defined by \eqref{eq:scheme_hats} - \eqref{eq:trace1},
 the mass, $L^2$-energy and Hamiltonian conservation properties in Lemma \ref{lemma:conserv_general} hold. 

\end{theorem} 

\begin{proof}
(i) The numerical traces in \eqref{eq:scheme_hats} are single-valued, so the DG scheme conserves the mass of the approximate solutions.

(ii) Using  \eqref{eq:scheme_hats},  we see that
\begin{alignat*}{1}
&\sum_{i=1}^N  \Big(\jmp{V(u_h)} - \{\varPi f(u_h)\} \jmp{u_h} +(\jmp{\varPi f(u_h)} - \jmp{p_h}) (\widehat{u_h} - \{u_h\})\\
& \quad\quad-\jmp{u_h}(\widehat{p_h} - \{p_h\}) + \varepsilon\jmp{q_h} (\widehat{q_h} - \{q_h\}) \Big)(x_i)\\
 = &\sum_{i=1}^N  \left(\jmp{V(u_h)} - \{\varPi f(u_h)\} \jmp{u_h}  -  \tau_{pu} \jmp{u_h}^2 + \varepsilon\tau_{qu}\jmp{u_h} \jmp{q_h} \right),
\end{alignat*}
which is equal to $0$ when the condition \eqref{eq:tau_pu1} 
holds. Then we get the $L^2$ conservation by Lemma \ref{lemma:conserv_general}.

(iii) Using the definition of the numerical traces \eqref{eq:scheme_hats}, we get
\begin{alignat*}{1}
& \sum_{i=1}^N  \left( \jmp{p_h} (\widehat{p_h} - \{p_h\}) +\varepsilon \jmp{q_h} (\widehat{u_h} - \{u_h\})_t + \varepsilon \jmp{u_h}_t (\widehat{q_h} - \{q_h\})\right)(x_i)\\
= & \sum_{i=1}^N  \left( \tau_{pu} \jmp{p_h} \jmp{u_h} + \tau_{qu} \jmp{u_h}_t \jmp{u_h}\right)=0
\end{alignat*}
by \eqref{eq:tau_qu1}. So we immediately get the conservation of the Hamiltonian \eqref{eq:Hamiltonian_conserv} using Lemma \ref{lemma:conserv_general}. 

This concludes the proof of Theorem \ref{thm:conserv}.
\end{proof}

\begin{remark}\label{remark:choice_hats}
	We would like to point out that Lemma \ref{lemma:conserv_general} provides a framework for achieving full conservation of mass, energy and Hamiltonian. Specifically, any choices of $\widehat{u_h}, \widehat{q_h}, \widehat{p_h}$ that satisfy the conditions  \eqref{eq:conserv_cond1} and \eqref{eq:conserv_cond2} will do. The numerical traces we have in \eqref{eq:scheme_hats} are just one of them. There are many other choices. For example, one can choose $\widehat{q_h}=\{q_h\}$ and determine $\widehat{u_h}$ and $\widehat{p_h}$ from equations \eqref{eq:conserv_cond1} and \eqref{eq:conserv_cond2}. 
	The scope of this paper is to discover a novel  paradigm for designing new conservative DG methods by letting the stabilization parameters be new unknowns so that conservation properties can be  explicitly embedded into the scheme and therefore their achievement guaranteed. 

\end{remark}

\section{Implementation}
\label{sec:implementation}

In this section, we provide a high-level summary of the implementation of our method. Further details are deferred to Appendix \ref{appendix:Implementation}.

\subsection{Time-stepping scheme}
Since KdV equations have the third-order spatial derivative term, we choose implicit time-marching schemes to avoid using extremely small time steps. Moreover, we need the time stepping method to be conservative so that the fully discrete scheme is conservative. 
Here,
we use the following implicit second-order Midpoint method, which preserves the conservation laws up to round-off error. \dong{This is proven in  \cite{DekkerVerwer1984} and adopted in \cite{bona2013conservative,KarakashianXing16} for the development of energy-conserving DG methods and \cite{liu2016hamiltonian} for a Hamiltonian-preserving DG scheme. Numerical results therein and of our paper demonstrate numerically that the Midpoint method does indeed conserve conservation laws including Hamiltonian.} Let $0=t_0< t_1<\cdots<t_M=T$ be a uniform partition of the interval $[0, T]$ and $\Delta t= t_{n+1}-t_n$ be the step size. For $n=0, \dong{\ldots}, M-1$, let $u_h^{n+1}\in W^k_h$ be defined as:
\[u_h^{n+1} = 2u_h^{n+\frac{1}{2}}-u_h^{n},\]
where $u_h^{n+\frac{1}{2}}\in W^k_h$ is the DG solution to the equation
\[\frac{u-u_h^n}{\frac{1}{2}\Delta t}+\varepsilon u_{xxx}+f(u)_x=g(x,t_{n+\frac{1}{2}}).
\]

 At every time step $t_{n+\frac{1}{2}}, n=0, \dong{\ldots}, M-1$, we need to solve equations \eqref{eq:scheme_time}, \eqref{eq:tau_pu1}, and \eqref{eq:tau_qu1} for $u_h$, $q_h$, $p_h$, $\tau_{qu}$, and $\tau_{pu}$. We can rewrite the nonlinear system into the following matrix-vector form and use MATLAB's built-in function ``Fsolve" to solve it.
\begin{subequations}\label{MatrixEqsF}
	\begin{align}
		&\textbf{M}[q]+(\textbf{D}+\textbf{A})[u]  = 0 \label{eq:F1}\\
		&\textbf{M}[p] + \varepsilon (\textbf{D}+ \textbf{A})[q] + \varepsilon \tau_{qu}\textbf{J}[u]  - \textbf{M}[f(u_h))] = 0\label{eq:F2}\\
		&\textbf{M}[u] - \frac{1}{2}\Delta t(\textbf{D}+  \textbf{A})[p]  -\frac{1}{2}\Delta t\tau_{pu}\textbf{J}[u] -\textbf{M}[\bar{u}]-\frac{1}{2}\Delta t \textbf{M}[g]  = 0\label{eq:F3}\\
		& V_f - \tau_{pu} \eta(u_h,u_h) + \varepsilon \tau_{qu} \eta(q_h, u_h) = 0\label{eq:F4}\\
		&\tau_{pu}\eta(p_h, u_h) + \tau_{qu}\sum_{i=1}^N \varepsilon \jmp{u_h}\jmp{u_h}_{t}(x_i)   = 0\label{eq:F5} 
	\end{align}
\end{subequations}
where $[u], [q], [p]$ are vectors consisting of degrees of freedom of $u_h^{n+ \frac{1}{2}}, q_h^{n+ \frac{1}{2}}, p_h^{n+ \frac{1}{2}}$, respectively, $[\bar{u}]$ is the  {\em known} vector for the degrees of freedom of $u_h^n$,  $\textbf{M}$ is the mass matrix, $\textbf{D}$ is the derivative matrix,  $\textbf{A}$ is \dong{the} matrix associated to the average flux, and $\textbf{J}$ is the matrix associated to the jump; see Appendix \ref{appendix:Implementation} for details on these matrices. 
 In \eqref{eq:F4} and \eqref{eq:F5}, we have adopted the notation defined in Section \ref{sec:dgm} 
\[
\eta(w,v)=\sum_{i=1}^N\jmp{w} \jmp{v}(x_i) \quad \textrm{ for any } w, v \in \big \{u_h,q_h,p_h\big\},
\]
and a new quantity 
$V_f := {\displaystyle \sum_{i=1}^N (\jmp{V(u_h)} - \{\varPi f(u_h)\} \jmp{u_h})(x_i)}.$ 

 The solution of this system, $([u], [q], [p], \tau_{qu}, \tau_{pu})$, can be considered as a column vector of size $[3(N-1)(k+1) + 2]$. So by introducing two more unknowns ($\tau_{qu}, \tau_{pu}$) and enforcing the two equations for conservation of energy and Hamiltonian, we only increase the size of the system by 2.

\subsection{Three-point difference formulas for $\jmp{u_h}_t$}
The last equation of the system, \eqref{eq:F5}, contains the non-traditional term $\jmp{u_h}_t$. We approximate it by the following three-point difference formula on uniform stencil to maintain the second-order accuracy in time

\[
\jmp{u_h}_t^{n+ \frac{1}{2}}=\frac{1}{\Delta t}\big(\jmp{u_h}^{n- \frac{1}{2}}-4\jmp{u_h}^n+3\jmp{u_h}^{n+ \frac{1}{2}}\big)+\mathcal{O}(\Delta t ^2).
\]
When $n=0$, we approximate $\jmp{u_h}_t^{\frac{1}{2}}$ by a three-point difference formula on a non-uniform stencil using 
$\jmp{u_h}$ at $t=0, (\frac{\Delta t}{2})^2,$ and $\frac{\Delta t}{2}$,
where $u_h^0$ is obtained by the $L^2$-projection of $u_0$ and $u_h^{(\frac{\Delta t}{2})^2}$ is computed using the backward Euler method. The nonuniform three-point difference formula for $\jmp{u_h}_t^{\frac{1}{2}}$ is as follows: 
\[
\jmp{u_h}_t^{\frac{1}{2}}=c_1 \jmp{u_h}^0 +c_2 \jmp{u_h}^{(\frac{\Delta t}{2})^2} +c_3\jmp{u_h}^{\frac{\Delta t}{2}}+\mathcal{O}(\Delta t^2)
\]
where 
\[
c_1=\frac{1-\frac{\Delta t}{2}}{\big(\frac{\Delta t}{2}\big)^2},\quad c_2=-\frac{1}{\big(\frac{\Delta t}{2}\big)^2 \big(1-\frac{\Delta t}{2}\big)},\quad c_3=\frac{2-\frac{\Delta t}{2}}{\big(\frac{\Delta t}{2}\big)\big(1-\frac{\Delta t}{2}\big)}.
\]

\subsection{The flowchart of the whole algorithm}
After solving for $u_h^{n+\frac{1}{2}}$ from the system \eqref{MatrixEqsF} with the $\jmp{u_h}_t$ term approximated by the three-point difference formulas above, we compute $u_h^{n+1}$ through the midpoint method. Then we solve for $q_h^{n+1}$ from the linear equation \eqref{eq:F1} using $u_h^{n+1}$. In order to obtain $p_h^{n+1}$, $\tau_{qu}^{n+1}$ and $\tau_{pu}^{n+1}$, we solve a smaller nonlinear system consisting of equations \eqref{eq:F2}, \eqref{eq:F4} and \eqref{eq:F5} using $u_h^{n+1}$ and $q_h^{n+1}$. To summarize, we use the following flowchart to describe the whole algorithm.\\

\begin{tikzpicture}[font=\small,thick]
	
	\node[draw,
	rounded rectangle,
	minimum width=2cm,
	minimum height=1cm] (block1) {\bf START};
	
	\node[draw,
	align=center,
	trapezium, 
	trapezium left angle = 55,
	trapezium right angle = 125,
	trapezium stretches,
	right=of block1,
	minimum width=1cm,
	minimum height=1cm
	] (block2) { Compute $u_h^0$ from $u_0$  \\ and set $n=0$. };
	
	\node[draw,
	align=center,
	right=of block2,
	minimum width=1.0cm,
	minimum height=1cm
	] (block3) { Use backward Euler method \\to  evaluate $u_h$ at $t=\big(\frac{\Delta t}{2}\big)^2$ };
	
	\node[draw,
	below =of block3,
	align=center,
	minimum width=2.5cm,
	minimum height=1.5cm
	] (block4) { Solve system \eqref{MatrixEqsF} for \\$u_h^{n+\frac{1}{2}}$ and use the Midpoint \\ method to get  $u_h^{n+1}$};
	
	\node[draw,
	left=of block4,
	align=center,
	minimum width=3.5cm,
	minimum height=1.5cm
	] (block5) {Solve \eqref{eq:F1} for $q_{h}^{n+1}$. \\ Solve \eqref{eq:F2}, \eqref{eq:F4} and \eqref{eq:F5} for\\ $\big(p_{h}^{n+1}, \tau_{qu}^{n+1}, \tau_{pu}^{n+1}\big)$.\\ 
	Let $n \leftarrow n+1$};
	
	\node[draw,
	diamond,
	below=of block5,
	minimum width=2.5cm,
	inner sep=0] (block6) { $t_{n} \ge T$ ?};
	
	\node[draw,
	rounded rectangle,
	left =of block6,
	minimum width=2cm,
	minimum height=1cm
	] (block7) {\bf End};
	
	\draw[-latex] (block1) edge (block2)
	(block2) edge (block3)
	(block3) edge (block4)
	(block4) edge (block5)
	(block5) edge (block6);
   \draw[-latex] (block6)--node[anchor=south,pos=0.5,fill=white,inner sep=3]{Yes} (block7);  
	\draw[-latex] (block6) -| (block4)
	node[anchor=south,pos=0.25,fill=white,inner sep=3]{No};
\end{tikzpicture}


\bigskip
\section{Numerical Results}
\label{sec:numericalresults}
In this section, we carry out numerical experiments to test the convergence and conservation properties of our DG method. \bo{In the first test problem, we consider a third-order linear equation with $f(u) = u$.}
In the second test problem, we use our DG method to solve a third-order nonlinear equation with $\varepsilon=1$, $0.1$, and $0.01$ and the solutions are sine waves that are periodic on the domain. In the \bo{last} test problem, we solve the classical KdV equation with a cnoidal wave solution and  compare the approximate solution with the exact one. 
For \dong{all the} test problems, we compute the $L^2$-errors and convergence orders and check the conservation of the energy and Hamiltonian of the DG solutions.

\subsection{Numerical Experiment 1}
In this test, we solve the following third-order linear equation  \bo{in \cite{zhang2019conservative}} 
\begin{equation*}
	u_{t} + \varepsilon u_{xxx} + (f(u))_{x} = 0,
\end{equation*}
where $\varepsilon = 1$ and $f(u) = u$,  with periodic boundary conditions on the domain $\Omega = [0,4\pi]$ and the initial condition $u_{0} = \sin(\frac{1}{2}x)$.  The exact solution to this problem is
\[u(x,t) = \sin\bigg(\frac{1}{2}x - \frac{3}{8}t\bigg).\] 
First, we test the convergence of the DG method for this linear problem. We use polynomials of degree $k=0, 1, 2$ for approximate solutions\dong{, the mesh size $h=\frac{4\pi}{N}$ for $N=2^l, l=3,\ldots, 7$, 
and  $\Delta t = \dong{0.2}(\frac{h}{4\pi})^{\min\{k,1\}}$ for time discretization.} The  $L^2$-errors and orders of convergence of the approximate solutions are displayed in Table \ref{fig:TableLinear} for  the final time $T = 0.1$. 
 \dong{We see that the approximate solutions for the variable $u$ converge with an optimal order for all polynomial degrees $k$, those for the auxiliary variable $q$ have an optimal convergence order for even $k$ and a sub-optimal order for odd $k$, and those for $p$ have sub-optimal orders for $k=1, 2$. }

\begin{table}
	\dong{
	\begin{tabular}{|l|c|c|c|c|c|c|c|c|}
		\hline
		\multirow{2}{*}{\textbf{k}}& \multirow{2}{*}{\textbf{N}}& \multicolumn{2}{c|}{$\mathbf{u_h}$} &\multicolumn{2}{c|}{$\mathbf{q_h}$} & \multicolumn{2}{c|}{$\mathbf{p_h}$} \\ 
		\cline{3-8}
		&  &\textbf{$L_{2}$ Error}&\textbf{Order }&\textbf{$L_{2}$ Error}&\textbf{Order }&\textbf{$L_{2}$ Error}&\textbf{Order }\\\hline
		\multirow{4}{*}{0}&\textbf{8}& 5.70e-1&-& 3.10e-1&-& 1.81e-0&-\\
		&\textbf{16}& 2.98e-1& 0.93& 1.52e-1& 1.02& 1.89e-0& -0.06\\
		&\textbf{32}& 1.42e-1& 1.07& 7.14e-2& 1.09& 1.07e-1& 4.15\\
		&\textbf{64}& 7.10e-2& 1.00& 3.56e-2& 1.01& 5.33e-2& 1.00\\
		&\textbf{128}& 3.55e-2& 1.00& 1.78e-2& 1.00& 2.66e-2&1.00\\\hline
		\multirow{4}{*}{1}&\textbf{8}& 5.80e-2&-& 2.55e-1&-& 2.34e-1&-\\
		&\textbf{16}& 1.44e-2& 2.01& 1.38e-1& 0.88& 1.35e-1& 0.79\\
		&\textbf{32}& 3.60e-3& 2.00& 7.06e-2&0.97 & 7.02e-2& 0.95\\
		&\textbf{64}& 9.00e-4& 2.00& 3.55e-2& 0.99& 3.62e-2& 0.95\\
		&\textbf{128}& 2.25e-4& 2.00& 1.78e-2& 1.00& 1.13e-2& 1.68\\\hline
		\multirow{4}{*}{2}&\textbf{8}& 3.93e-3&-& 7.92e-3&-& 4.07e-2&-\\
		&\textbf{16}& 4.84e-4& 3.02& 9.76e-4&3.02 &9.45e-3 & 2.11\\
		&\textbf{32}& 6.00e-5& 3.01& 1.23e-4&2.99 & 2.31e-3& 2.03\\
		&\textbf{64}& 7.47e-6& 3.01& 1.52e-5& 3.02& 5.78e-4& 2.00\\
		\hline
	\end{tabular}
	\caption{Numerical Experiment 1 (third-order linear equation): Error and convergence order of $u_h$, $q_h$, and $p_h$}
	\label{fig:TableLinear}
}
\end{table}

Next, we test the conservation of the energy and Hamiltonian of the approximate solution using  polynomials of degree $k=2$ on 32 intervals for the final time $T = 50$. In Figure \ref{fig:linear}, we see that the Hamiltonian and energy of the approximate solution remain the same over the whole time period. \dong{The errors of the Hamiltonian and energy are very small, as shown on the second row of Figure \ref{fig:linear}.}

\begin{figure}[!h]
	\begin{center}
		\includegraphics[scale=0.25]{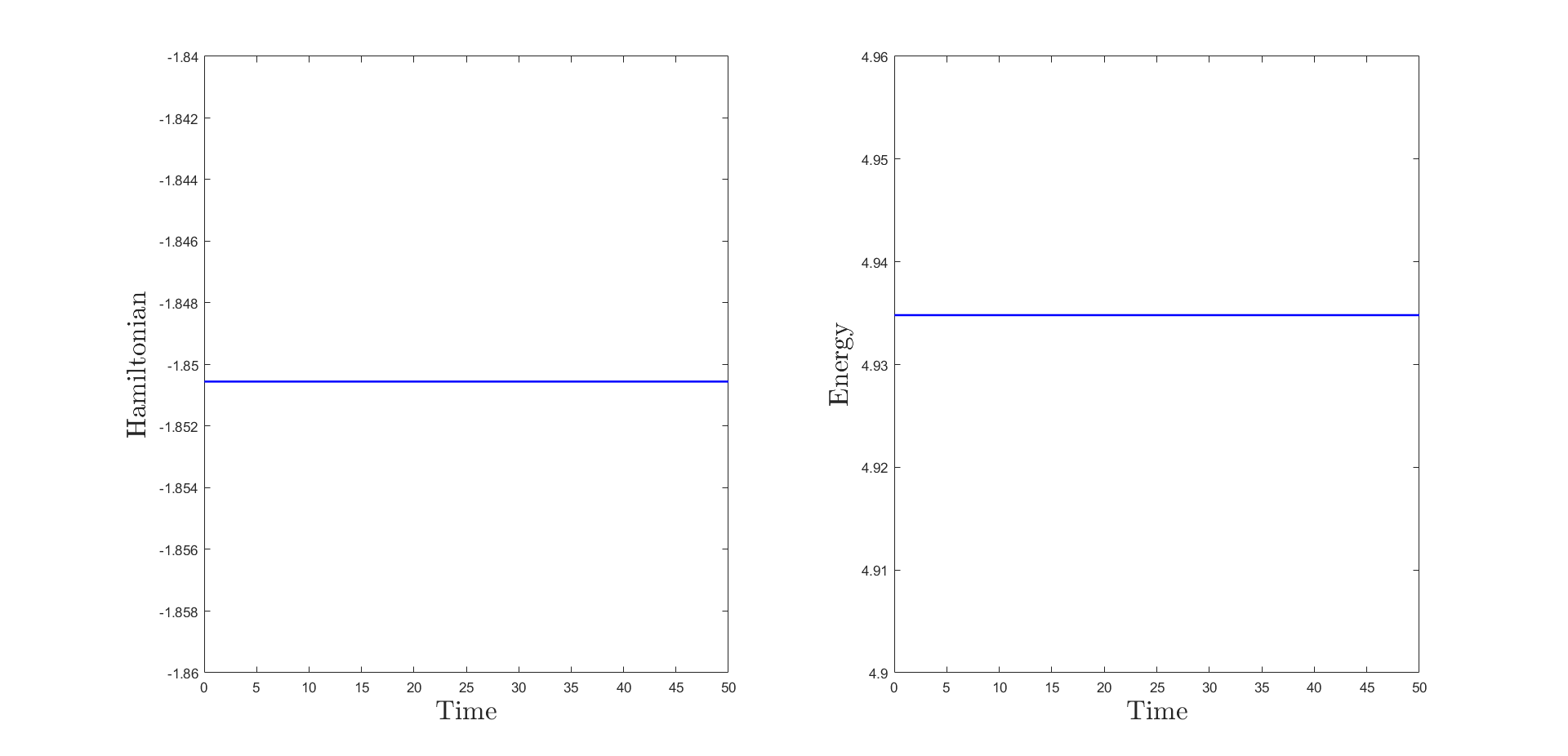}
		\includegraphics[scale=0.25]{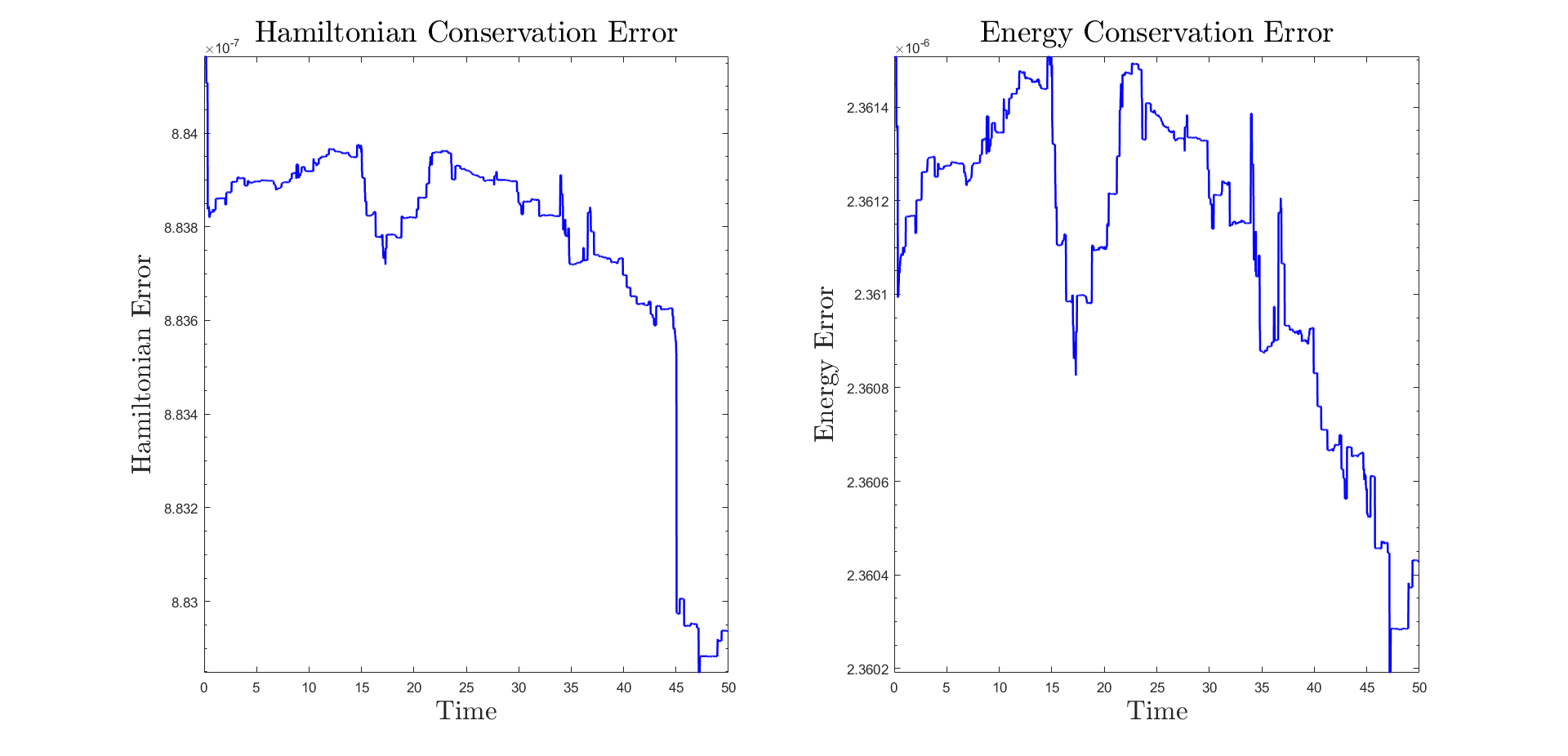}
	\end{center}
	\caption{Numerical Experiment 1 (third-order linear equation): Hamiltonian (Left) and energy (Right) conservation. \dong{Shown on the bottom are the corresponding errors.}}
	\label{fig:linear}
\end{figure}

\subsection{Numerical Experiment 2}
In the second test, we consider the following third-order nonlinear equation 
\begin{equation*}
    u_{t} + \varepsilon u_{xxx} + (f(u))_{x} = g
\end{equation*}
with periodic boundary conditions on $\Omega=[0, 1]$ and the initial condition $u_0=\sin{(2\pi x)}$, where $f(u) = \frac{u^2}{2}$ and $g$ is the function which gives the solution 
\[u(x,t) = \sin(2\pi x + t).\]

For this problem, we first test the convergence orders of our DG method for $\varepsilon=1$, $0.1$ and $0.01$ when using polynomials of degree $k=0, 1, 2$. We use \dong{$h=1/N$, where  $N=2^l, l=3, \ldots, 7$, and $\Delta t=0.2\, h^{\min\{k,1\}}$ for time discretization}, and the final time is $T=0.1$. The $L^2$-errors and orders of convergence   for $\varepsilon=1$, 0.1, 0.01  are displayed in Table \ref{fig:Table1}, Table \ref{fig:Table2}, and Table \ref{fig:Table3}, respectively. 
\dong{Note that for existing energy-conserving DG methods in  \cite{bona2013conservative,KarakashianXing16,chen2016new}, it is typical that approximate solutions to $u$ have optimal convergence orders when $k$ is even and sub-optimal orders when $k$ is odd. Here, we see that our method has comparable convergence rates.}

\begin{table}
	\dong{
	\begin{tabular}{|l|c|c|c|c|c|c|c|c|}
		\hline
		\multirow{2}{*}{\textbf{k}}& \multirow{2}{*}{\textbf{N}}& \multicolumn{2}{c|}{\bo{$\mathbf{u_h}$}} &\multicolumn{2}{c|}{\bo{$\mathbf{q_h}$}} & \multicolumn{2}{c|}{\bo{$\mathbf{p_h}$}} \\ 
		\cline{3-8}
		&  &\textbf{$L_{2}$ Error}&\textbf{Order }&\textbf{$L_{2}$ Error}&\textbf{Order }&\textbf{$L_{2}$ Error}&\textbf{Order }\\\hline
		\multirow{4}{*}{0}&\textbf{8}& 3.81e-1&-&1.96e-0 &-& 1.05e+1&-\\
		&\textbf{16}&8.00e-2 & 2.25& 5.14e-1 & 1.93 & 3.41e-0& 1.61\\
		&\textbf{32}&4.83e-2 & 0.73& 2.89e-1 &0.83 & 1.73e-0& 0.97\\
		&\textbf{64}&2.03e-2 & 1.25& 1.27e-1&1.18 & 7.99e-1& 1.12\\
		&\textbf{128}&1.01e-2 &1.01 & 6.33e-2&1.01 &3.97e-1 &1.01\\\hline
		\multirow{4}{*}{1}&\textbf{8}& 5.29e-2&-& 9.51e-1&-& 1.05e+1&-\\
		&\textbf{16}& 7.19e-2& -0.44& 8.16e-1& 0.22& 2.00e-0& 2.40\\
		&\textbf{32}& 1.67e-2& 2.10& 1.92e-1& 2.09& 3.78e-0& -0.92\\
		&\textbf{64}& 1.09e-3& 3.93& 1.26e-1& 0.61& 3.46e-2& 6.77\\
	    &\textbf{128}& 1.49e-4& 2.88& 6.29e-2& 1.00& 1.08e-1& -1.64\\\hline
		\multirow{4}{*}{2}&\textbf{8}& 2.08e-3&-& 1.23e-1&-& 7.89e-0&-\\
		&\textbf{16}& 1.35e-4& 3.94& 3.48e-3& 5.14 &4.27e-1 & 4.21\\
		&\textbf{32}& 1.69e-5& 3.00& 4.31e-4& 3.01 &1.04e-1 & 2.04\\
		&\textbf{64}& 2.11e-6& 3.00& 5.38e-5& 3.00 &2.59e-2 & 2.01\\
		\hline
	\end{tabular}
	\caption{Numerical Experiment \bo{2 (third-order nonlinear equation)}: Errors and convergence orders of $u_h$, $q_h$, and $p_h$ for $\varepsilon = 1$}
	\label{fig:Table1}
}
\end{table}

\begin{table}
	\dong{
	\begin{tabular}{|l|c|c|c|c|c|c|c|c|}
		\hline
		\multirow{2}{*}{\textbf{k}}& \multirow{2}{*}{\textbf{N}}& \multicolumn{2}{c|}{\bo{$\mathbf{u_h}$}} &\multicolumn{2}{c|}{\bo{$\mathbf{q_h}$}} & \multicolumn{2}{c|}{\bo{$\mathbf{p_h}$}} \\ 
		\cline{3-8}
		&  &\textbf{$L_{2}$ Error}&\textbf{Order }&\textbf{$L_{2}$ Error}&\textbf{Order }&\textbf{$L_{2}$ Error}&\textbf{Order }\\\hline
		\multirow{4}{*}{0}&\textbf{8}& 3.51e-1&-& 1.87e-0&-& 1.06e-0&-\\
		&\textbf{16}& 1.17e-1 & 1.58 & 6.75e-1 & 1.47 & 4.00e-1 & 1.40\\
		&\textbf{32}& 4.63e-2 & 1.34 & 2.80e-1 & 1.27 & 1.73e-1 & 1.21\\
		&\textbf{64}& 2.11e-2 & 1.14 & 1.31e-1 & 1.10 & 8.18e-2 & 1.08\\
		&\textbf{128}& 1.02e-2 & 1.05 & 6.36e-2 & 1.04 & 4.01e-2 &1.03\\\hline
		\multirow{4}{*}{1}&\textbf{8}& 6.62e-2&-& 8.79e-1&-& 1.17e-0&-\\
		&\textbf{16}& 4.08e-2& 0.70& 3.11e-1& 1.50& 7.86e-1& 0.58\\
		&\textbf{32}& 2.05e-2& 0.99& 1.48e-1& 1.07& 4.11e-1& 0.94\\
		&\textbf{64}& 1.48e-3& 3.80& 1.26e-1& 0.24& 5.84e-2& 2.82\\
	    &\textbf{128}& 3.71e-4& 2.00& 6.29e-2& 1.00& 5.14e-2& 0.18\\\hline
		\multirow{4}{*}{2}&\textbf{8}& 1.30e-3&-& 3.00e-2&-& 1.86e-1&-\\
		&\textbf{16}& 1.44e-4& 3.17& 3.52e-3& 3.09& 4.06e-2& 2.19\\
		&\textbf{32}& 1.69e-5& 3.09& 4.29e-4& 3.04& 1.04e-2 & 1.96\\
		&\textbf{64}& 2.11e-6& 3.00& 5.42e-5& 2.98& 2.63e-3& 1.99\\
\hline
	\end{tabular}
	\caption{Numerical Experiment \bo{2 (third-order nonlinear equation)}: Errors and convergence orders of $u_h$, $q_h$, and $p_h$ for $\varepsilon = 0.1$}
	\label{fig:Table2}
}
\end{table}

\begin{table}
	\dong{
		\begin{tabular}{|l|c|c|c|c|c|c|c|}
			\hline
			\multirow{2}{*}{\textbf{k}}& \multirow{2}{*}{\textbf{N}}& \multicolumn{2}{c|}{\bo{$\mathbf{u_h}$}} &\multicolumn{2}{c|}{\bo{$\mathbf{q_h}$}} & \multicolumn{2}{c|}{\bo{$\mathbf{p_h}$}} \\ 
			\cline{3-8}
			&  &\textbf{$L_{2}$ Error}&\textbf{Order }&\textbf{$L_{2}$ Error}&\textbf{Order }&\textbf{$L_{2}$ Error}&\textbf{Order }\\\hline
			\multirow{4}{*}{0}&\textbf{8}& 1.75e-1&-& 1.20e-0&-& 1.43e-1&-\\
			&\textbf{16}& 8.33e-2& 1.07& 5.62e-1& 1.10& 6.19e-2& 1.20\\
			&\textbf{32}& 4.07e-2& 1.03& 2.62e-1& 1.10& 2.73e-2& 1.18\\
			&\textbf{64}& 2.01e-2& 1.02& 1.27e-1& 1.05& 1.30e-2& 1.08\\
			&\textbf{128}& 1.00e-2& 1.00& 6.31e-2& 1.01& 6.40e-3&1.02\\\hline
					\multirow{4}{*}{1}&\textbf{8}& 2.30e-2&-& 8.93e-1&-& 9.15e-2&-\\
			&\textbf{16}& 4.08e-2& -0.83& 3.14e-1& 1.51& 7.09e-2&0.37 \\
			&\textbf{32}& 2.05e-3& 4.31& 2.53e-1& 0.31& 3.20e-2& 1.15\\
			&\textbf{64}& 1.04e-3& 0.99& 1.26e-1& 1.01& 1.59e-2& 1.01\\
			&\textbf{128}&5.89e-4 &0.81 & 6.29e-2 & 1.00 & 8.34e-3 &0.93 \\\hline
			\multirow{4}{*}{2}&\textbf{8}& 1.24e-3&-& 3.23e-2&-& 1.63e-2&-\\
			&\textbf{16}& 1.41e-4& 3.13& 3.39e-3& 3.25& 4.12e-3& 1.99\\
			&\textbf{32}& 1.70e-5& 3.06& 7.20e-4& 2.23& 1.75e-3& 1.24\\
			&\textbf{64}& 2.37e-6& 2.84& 1.08e-4& 2.74& 5.30e-4& 1.72\\
			\hline
		\end{tabular}
		\caption{Numerical Experiment \bo{2 (third-order nonlinear equation)}: Errors and convergence orders of $u_h$, $q_h$, and $p_h$ for $\varepsilon = 0.01$}
		\label{fig:Table3}
	}
\end{table}

Next, we plot the exact solutions and the numerical solutions with quadratic polynomials on 32 elements for different $\varepsilon$. Note that $u$ and $q$  are not changing with respect to $\varepsilon$ in this test problem, but $p$ depends on $\varepsilon$. So we plot $u, q$, $u_h$ and $q_h$ over the time period [0, 5] in Figure \ref{fig:graph5a} and the snapshot of them at the time $T=5$ in Figure \ref{fig:graph5a_finaltime}. The graphs of $p$ and $p_h$ for different $\varepsilon$ over the time period [0, 5] are plotted in Figure \ref{fig:graph5b} and the snapshots of them at the time $T=5$ are in Figure \ref{fig:graph5b_finaltime}. We see that in all the figures the graphs of numerical solutions match well with those of exact solutions.

\begin{figure}[!h]
\begin{center}
\includegraphics[scale=0.25]{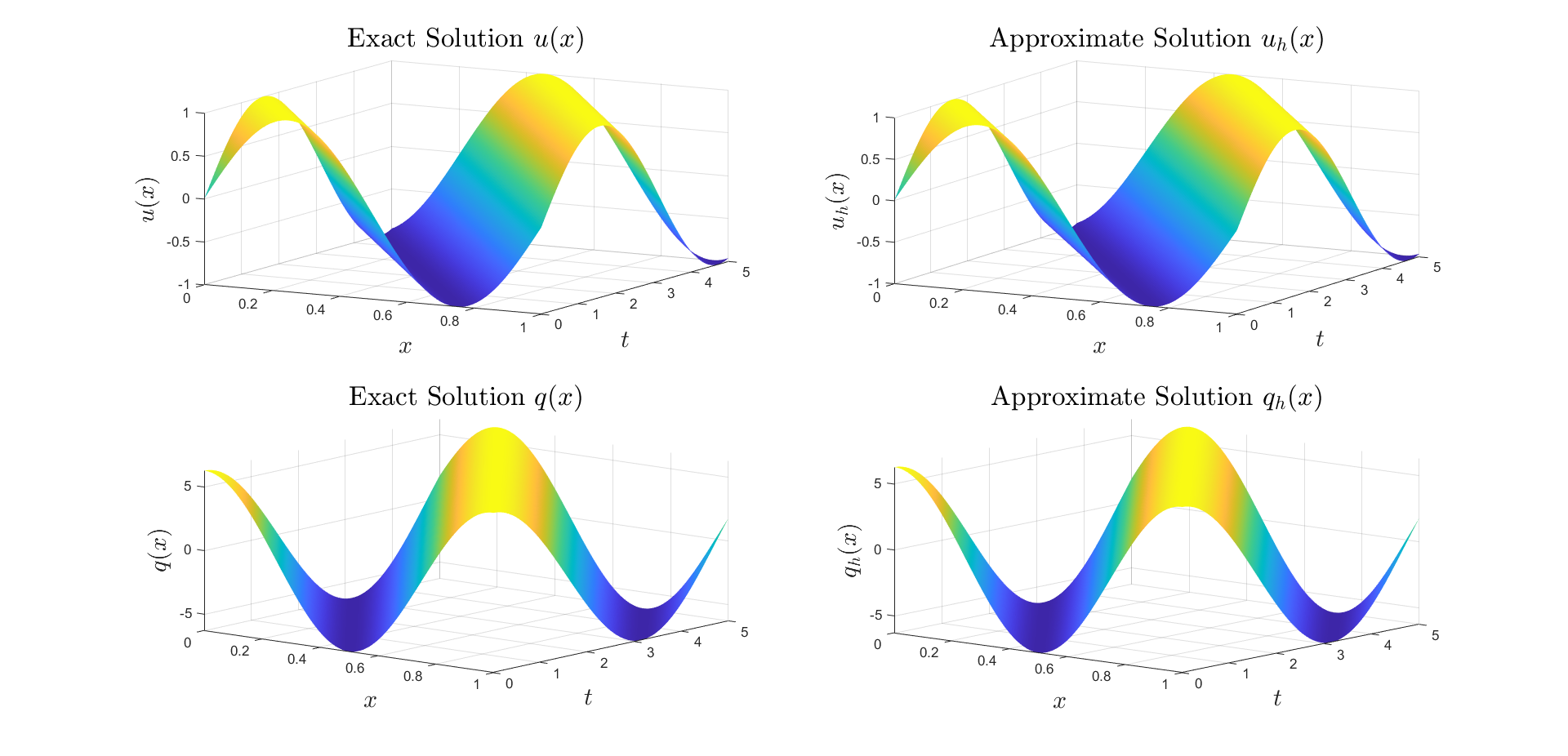}
\end{center}
	\caption{Numerical Experiment 2 (third-order nonlinear equation): Solutions in time (Left: exact solution, Right: approximate solution) for the $\varepsilon$-independent $u$ and $q$.}
\label{fig:graph5a}
\end{figure}

\begin{figure}[!h]
	\begin{center}
		\includegraphics[scale=0.25]{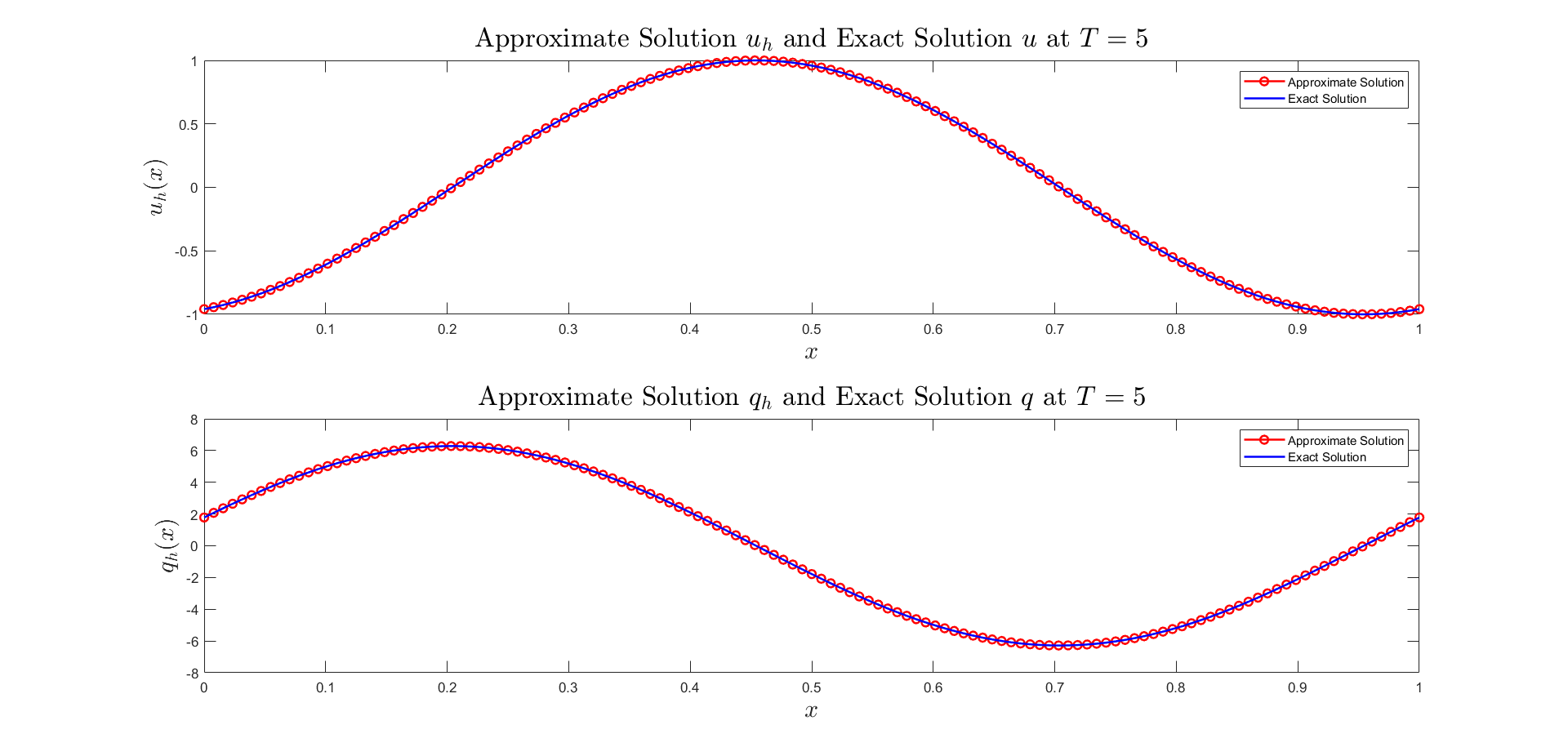}
	\end{center}
\caption{\dong{Numerical Experiment 2 (third-order nonlinear equation): Solutions at the final time $T=5$ (Top: $u$ and $u_h$, bottom: $q$ and $q_h$).}}
	\label{fig:graph5a_finaltime}
\end{figure}

\begin{figure}[!h]
\begin{center}
\includegraphics[scale=0.25]{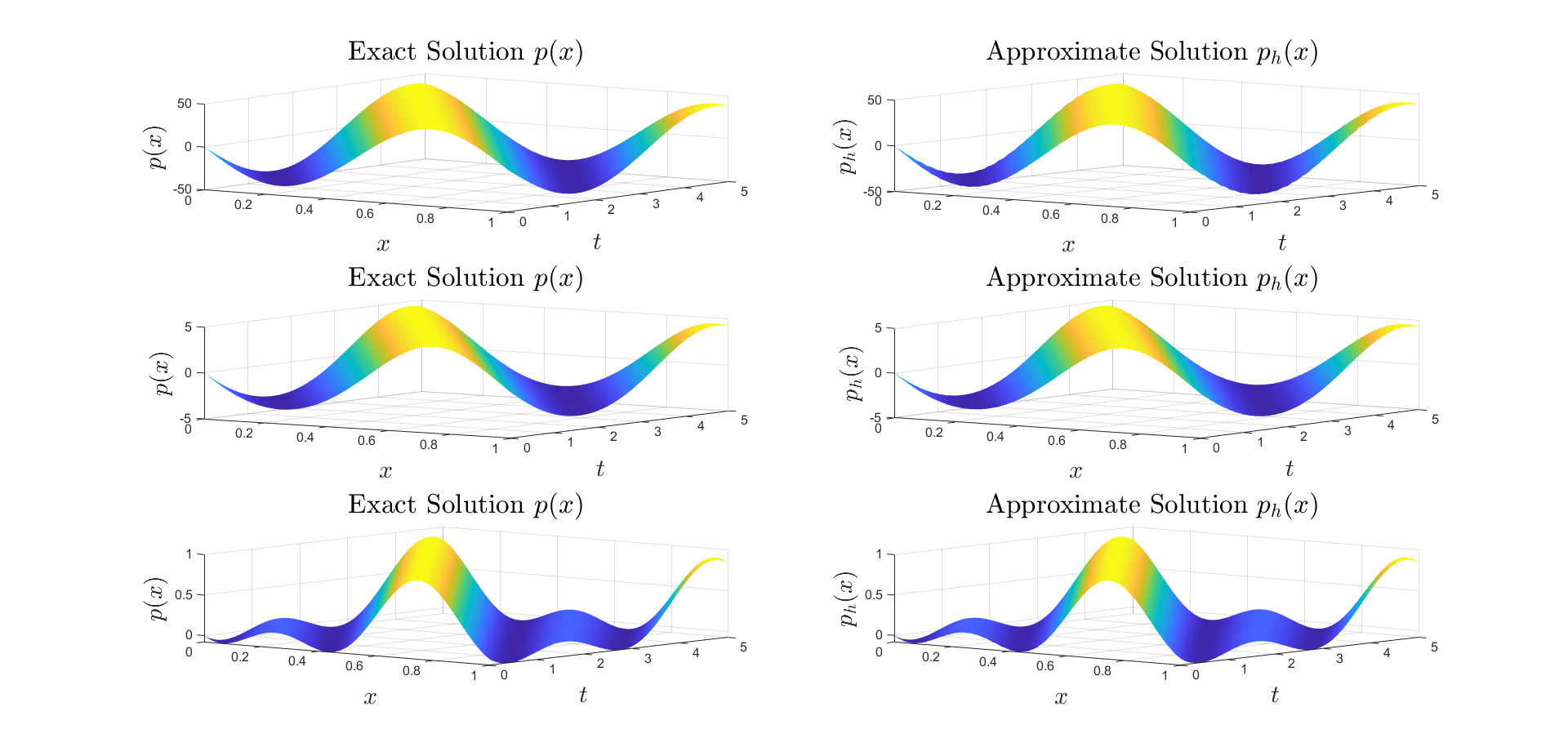}
\end{center}
\caption{{Numerical Experiment 2 (third-order nonlinear equation): Solution in time (Left: exact, Right: approximate, $\varepsilon = 1, 0.1, 0.01$ from top to bottom) for the $\varepsilon$-dependent $p$.}}
\label{fig:graph5b}
\end{figure}

\begin{figure}[!h]
\begin{center}
\includegraphics[scale=0.25]{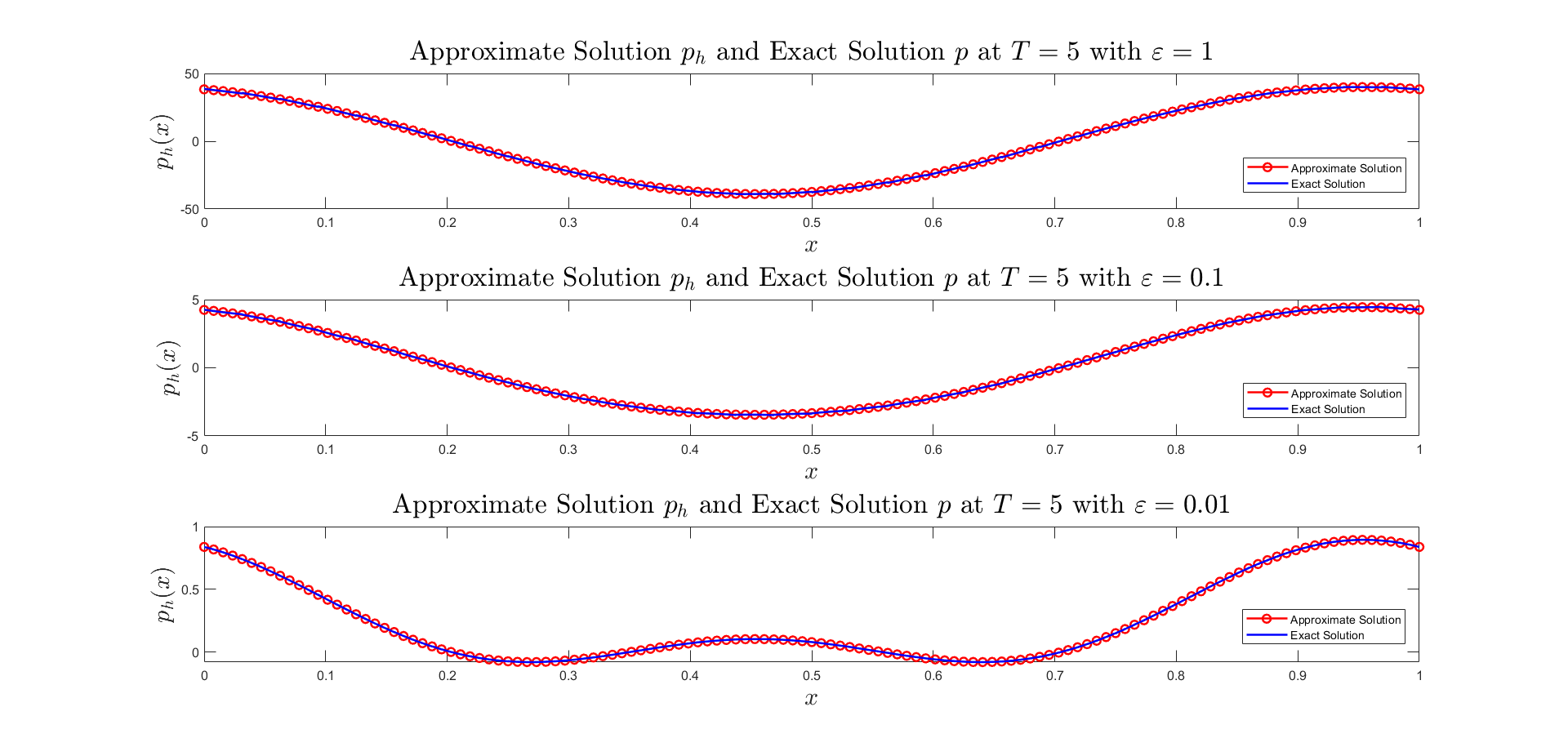}
\end{center}
\caption{\dong{Numerical Experiment 2 (third-order nonlinear equation): the $\varepsilon$-dependent solution $p$ and the approximate solution $p_h$ at the final time $T=5$ (with $\varepsilon = 1, 0.1, 0.01$ from top to bottom).}}
		\label{fig:graph5b_finaltime}
\end{figure}


Finally, we test the conservation properties of our DG scheme. We plot the Hamiltonian and the Energy of the numerical solutions for $t\in [0, 50]$  for different $\varepsilon$ in Figure \ref{fig:graph4}. \dong{The errors of the energy and Hamiltonian for different $\varepsilon$ are plotted in Figure \ref{fig:graph4_error_eps1}}.  We see that our method successfully conserves both Hamiltonian and energy. We note that, even though 
the energy and Hamiltonian are conserved for the KdV equations (i.e.\dong{,} the source term $g \equiv 0$), the manufactured solution of this particular test with a nonzero source term  happens to bear these properties as well and thus serves as an ideal test case. 

\begin{figure}[!h]
\begin{center}
\includegraphics[scale=0.25]{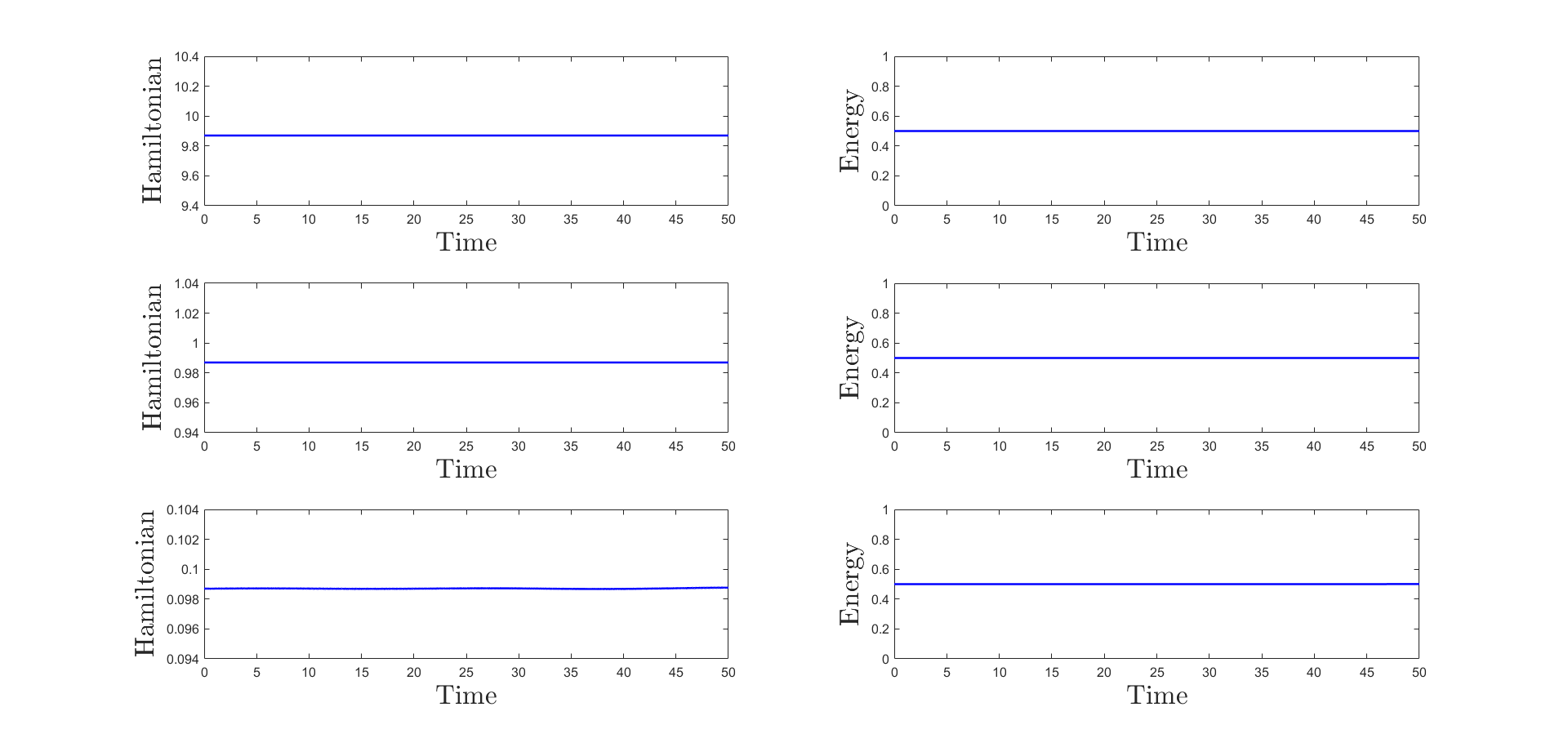}
\end{center}
\caption{ Numerical Experiment 2 (third-order nonlinear equation): Conservation of Hamiltonian (Left) and energy (Right)  when $\varepsilon$ = 1 (top), 0.1 (middle), and 0.01 (bottom).}
\label{fig:graph4}
\end{figure}

\begin{figure}[!h]
\begin{center}
\includegraphics[scale=0.25]{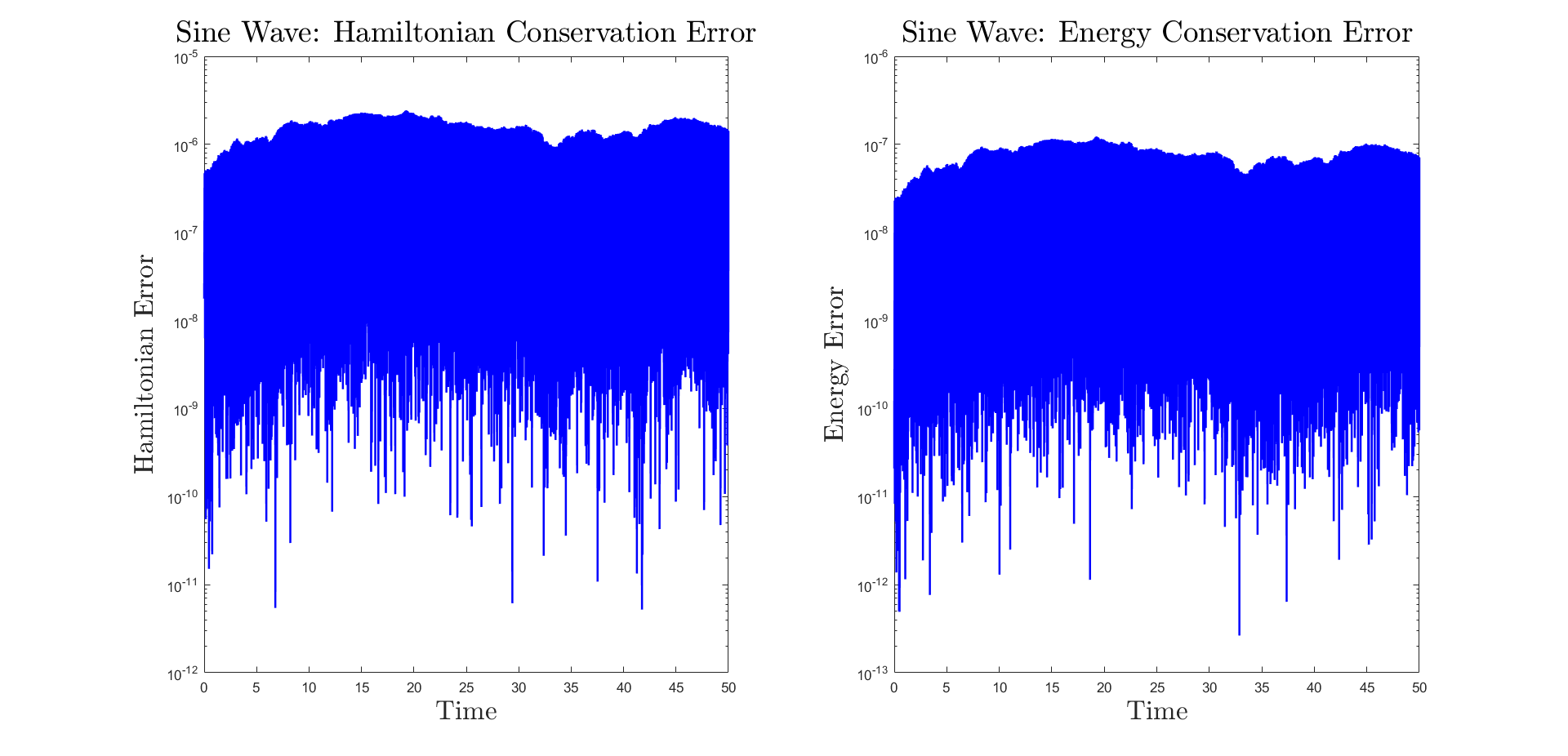}
\includegraphics[scale=0.25]{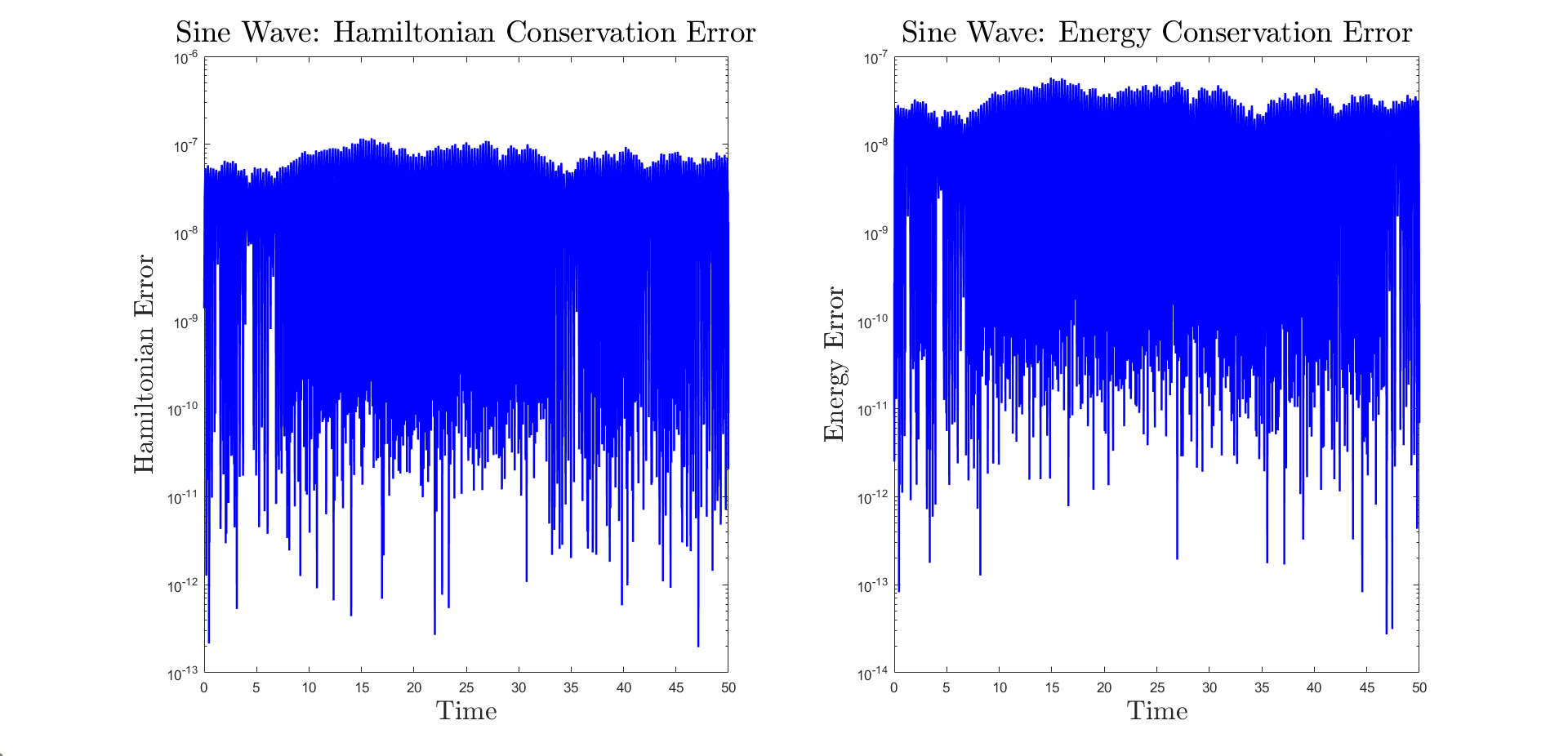}
\includegraphics[scale=0.25]{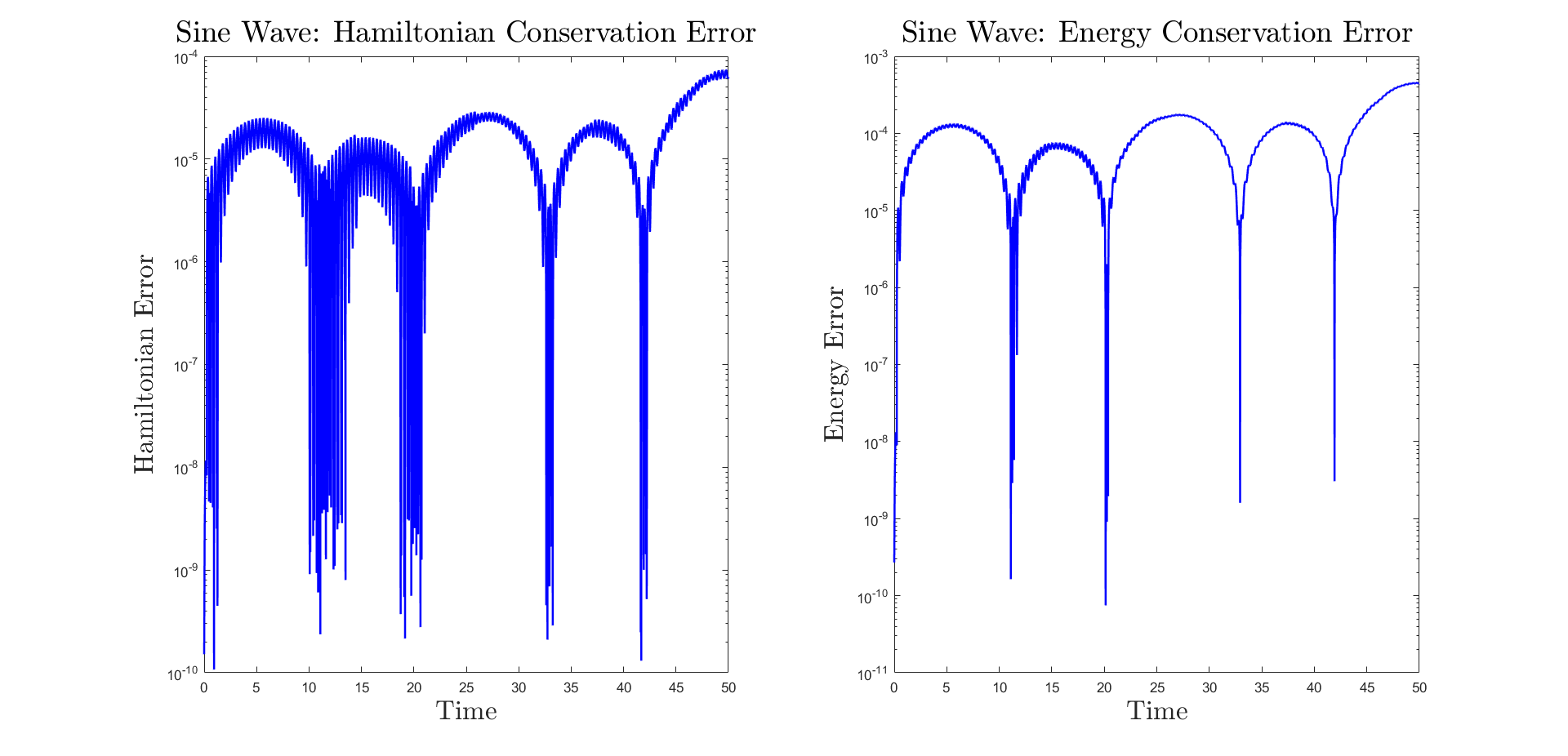}
\end{center}
\caption{\dong{Numerical Experiment 2 (third-order nonlinear equation): Errors of Hamiltonian (Left) and energy (Right) when $\varepsilon = 1$ (top), $0.1$ (middle) and $0.01$ (bottom).}}
\label{fig:graph4_error_eps1}
\end{figure}

\subsection{Numerical Experiment 3}
In this example, we test the KdV equation 
$$u_{t} + \varepsilon u_{xxx} + (f(u))_x = 0$$
with $\varepsilon = \frac{1}{24^2}$  and  $f(u) = \frac{u^2}{2}$. The domain is $\Omega = [0,1]$ and we are testing a cnoidal-wave solution 
 \[u(x,t) = Acn^2(z),\]
where $cn(z) = cn(z|m)$ is the Jacobi elliptic function with modulus $m = 0.9$, $z = 4K(x-vt-x_0)$, $A = 192m\varepsilon K(m)^2$, $v = 64\varepsilon (2m-1)K(m)^2$, and $K(m)=\int_{0}^{\frac{\pi}{2}} \frac{d\theta}{\sqrt{(1-msin^2\theta)}}$ is the Jacobi elliptic integral of the first kind; see \cite{abramowitz1970handbook}. The parameter $x_0$ is arbitrary, so we take it to be zero. The solution $u$ has a spatial period 1.

This benchmark problem has been tested for other conservative DG methods in \cite{bona2013conservative,KarakashianXing16,liu2016hamiltonian,zhang2019conservative}. Those methods conserve either the Hamiltonian or the energy of the solution, \dong{but not both.}

\dong{In Table \ref{fig:Table4}, we display the $L^2$ errors  of approximate solutions to $u, q$, and $p$ for $k=0, 1, 2$. The convergence orders are similar to those in the previous  numerical experiments.} In Figure \ref{fig:graph8}, we plot the exact solution and the approximate solution using polynomial degree $k=2$ over 32 intervals over the time period $t\in[0, 5]$. \dong{The snapshots of the exact and the approximation solutions at the final time $T=5$ are shown in Figure \ref{fig:graph8_finaltime}.}   We can see that the graphs of exact solution and the approximate solution match up well \dong{in both figures}.
Next, we compute the numerical solution using $k=2$ on 32 intervals for a longer time $T=50$. The graphs of the Hamiltonian and energy of the DG solution versus time are displayed in Figure \ref{fig:graph9}, \dong{and the errors of Hamiltonian and energy are plotted on the second row of Figure \ref{fig:graph9}}. We can see that both the Hamiltonian and the energy have been conserved during the whole time period.

\begin{table}
	\dong{
	\begin{tabular}{|l|c|c|c|c|c|c|c|}
		\hline
		\multirow{2}{*}{\textbf{k}}& \multirow{2}{*}{\textbf{N}}& \multicolumn{2}{c|}{\bo{$\mathbf{u_h}$}} &\multicolumn{2}{c|}{\bo{$\mathbf{q_h}$}} & \multicolumn{2}{c|}{\bo{$\mathbf{p_h}$}} \\ 
		\cline{3-8}
		&  &\textbf{$L_{2}$ Error}&\textbf{Order }&\textbf{$L_{2}$ Error}&\textbf{Order }&\textbf{$L_{2}$ Error}&\textbf{Order }\\\hline
		\multirow{4}{*}{0}&\textbf{8}& 5.44e-1&-& 8.11&-& 3.18e-1&-\\
		&\textbf{16}& 2.72e-1& 1.00& 4.56& 0.83& 2.42e-1& 0.40\\
		&\textbf{32}& 9.89e-2& 1.46& 1.65& 1.47& 8.80e-2& 1.46\\
		&\textbf{64}& 4.50e-2& 1.14& 7.69e-1& 1.10& 3.05e-2& 1.53\\
	    &\textbf{128}& 2.22e-2& 1.02& 3.87e-1& 0.99& 1.34e-2& 1.19\\\hline
		\multirow{4}{*}{1}&\textbf{8}& 1.26e-1&-&2.40 &-& 1.12e-1&-\\
		&\textbf{16}& 7.49e-2& 0.75& 2.82& -0.23& 1.51e-1& -0.43\\
		&\textbf{32}& 2.13e-2& 1.81& 1.41& 1.00& 7.23e-2& 1.06\\
		&\textbf{64}& 6.07e-3& 1.81& 7.41e-1& 0.93& 5.01e-2& 0.53\\
	    &\textbf{128}& 1.57e-3& 1.95& 3.74e-1& 0.99& 2.88e-1& 0.80\\\hline
		\multirow{4}{*}{2}&\textbf{8}& 1.18e-1&-& 4.70&-& 2.58e-1&-\\
		&\textbf{16}& 1.60e-2& 2.88& 7.24e-1& 2.70& 8.08e-2& 1.68\\
		&\textbf{32}& 2.71e-3& 2.56& 5.96e-2& 3.60& 1.15e-2& 2.85\\
		&\textbf{64}& 3.47e-4& 2.96& 6.15e-3& 3.28& 6.08e-4& 4.21\\
		\hline
	\end{tabular}
	\caption{Numerical Experiment 3 (classical KdV equation): Errors and convergence orders of $u_h$, $q_h$, and $p_h$}
	\label{fig:Table4}
}
\end{table}

\begin{figure}[!h]
\begin{center}
\includegraphics[scale=0.25]{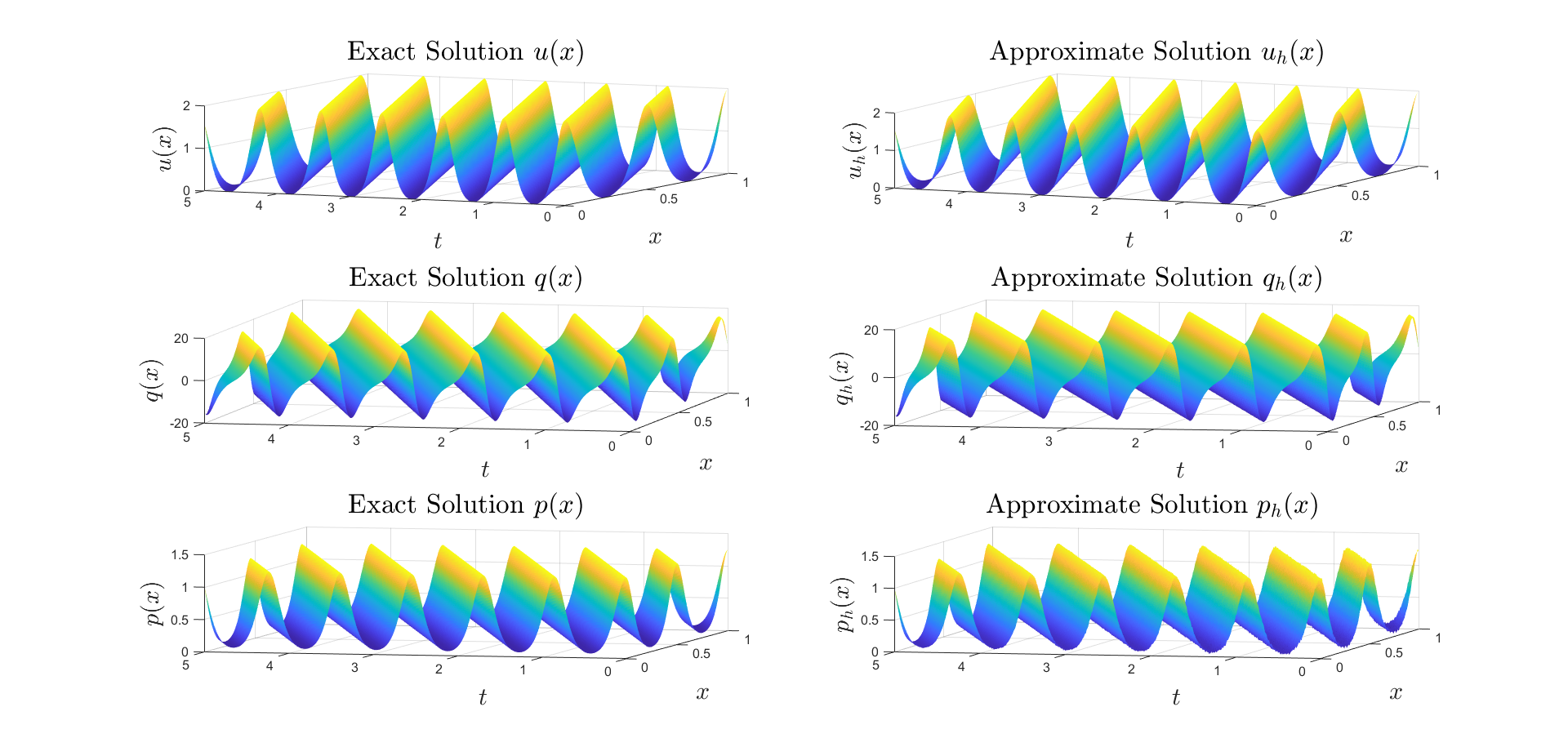}
\end{center}
\caption{Numerical Experiment \bo{3 (classical KdV equation)}: Solution in time (Left: exact, Right: approximate) for the Cnoidal Wave. }
\label{fig:graph8}
\end{figure}

\begin{figure}[!h]
\begin{center}
\includegraphics[scale=0.25]{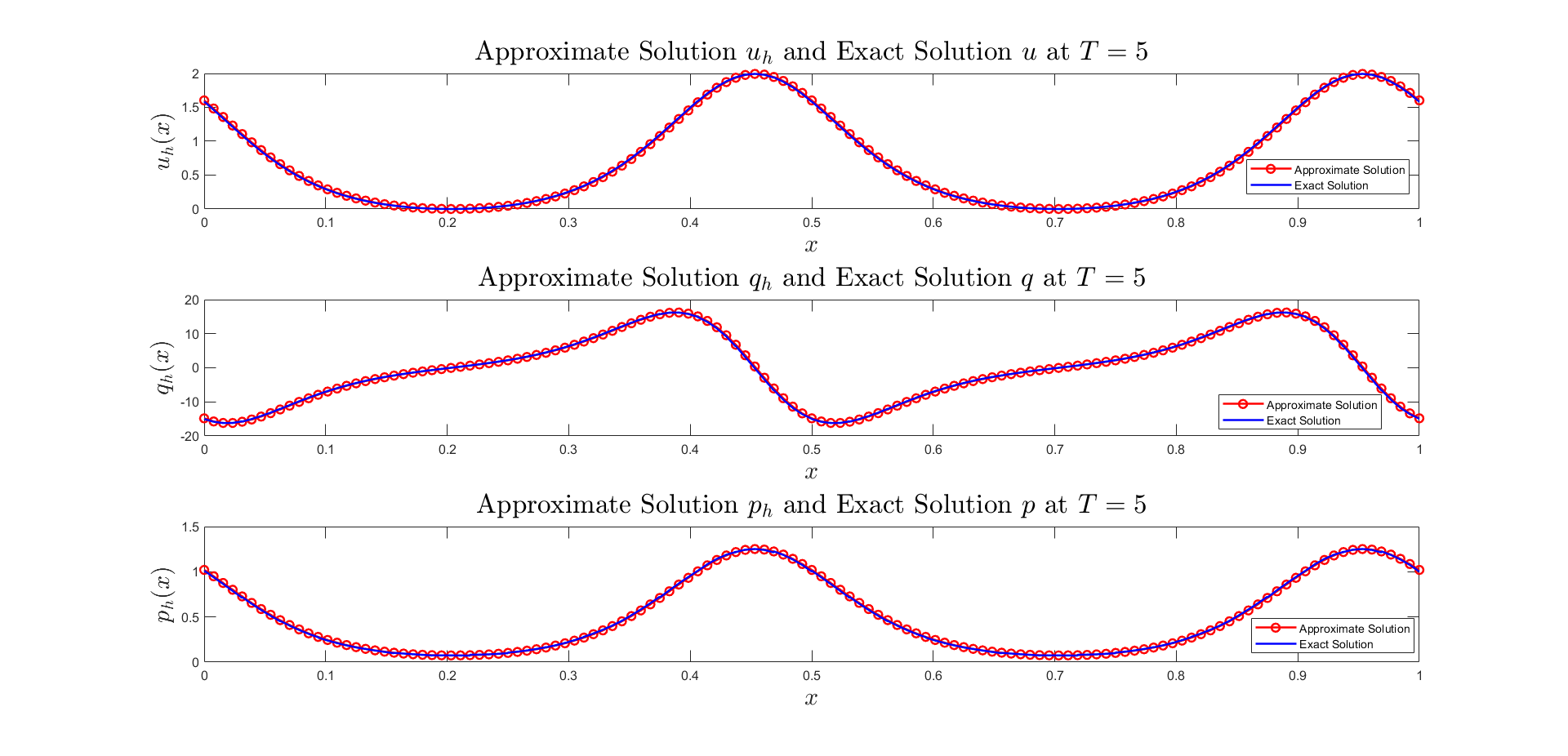}
\end{center}
\caption{\dong{Numerical Experiment 3 (classical KdV equation): Exact and approximate solutions at the final time $T=5$ for the Cnoidal Wave (Top: $u$ and $u_h$, middle: $q$ and $q_h$, bottom: $p$ and $p_h$).}}
\label{fig:graph8_finaltime}
\end{figure}

\begin{figure}[!h]
\begin{center}
\includegraphics[scale=0.25]{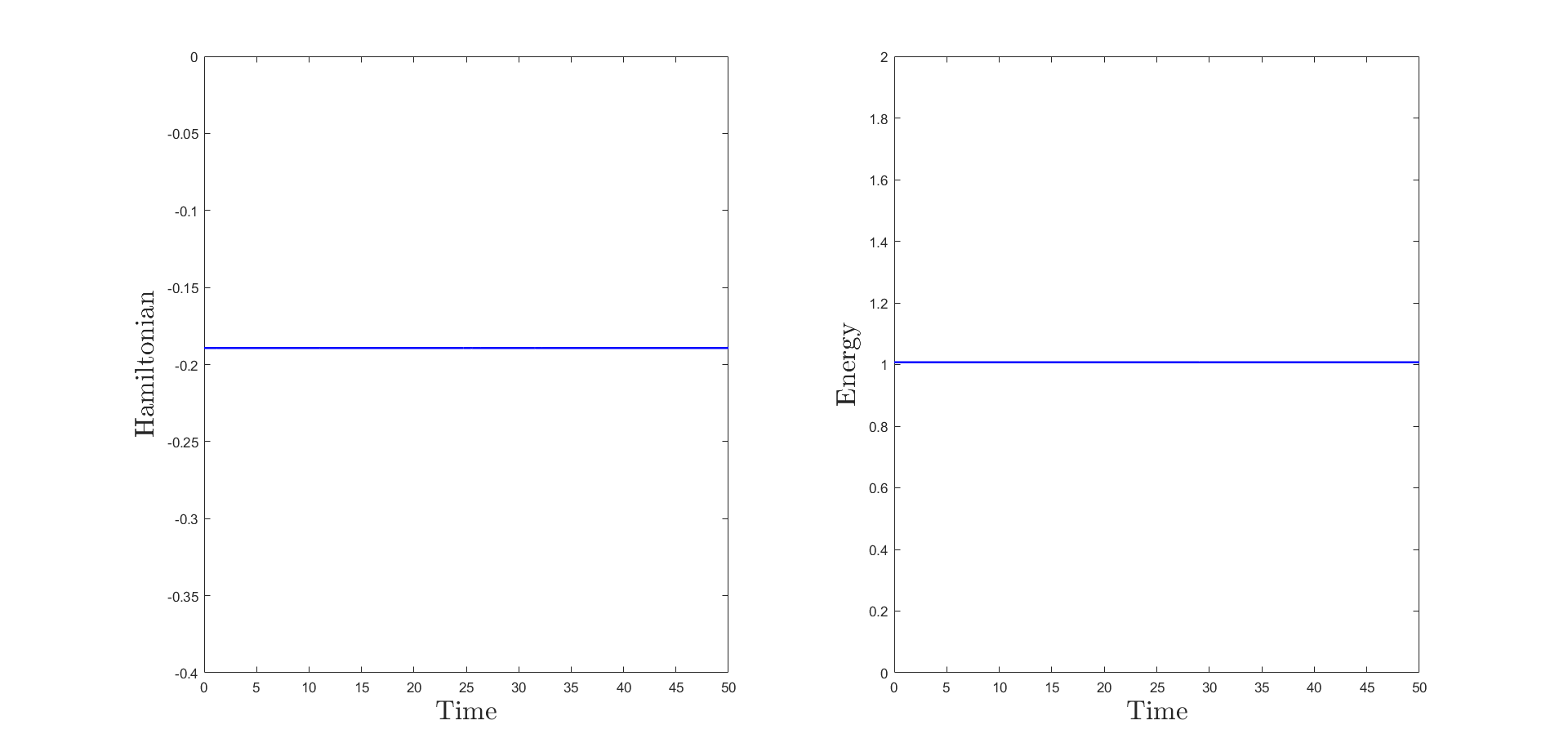}
\includegraphics[scale=0.25]{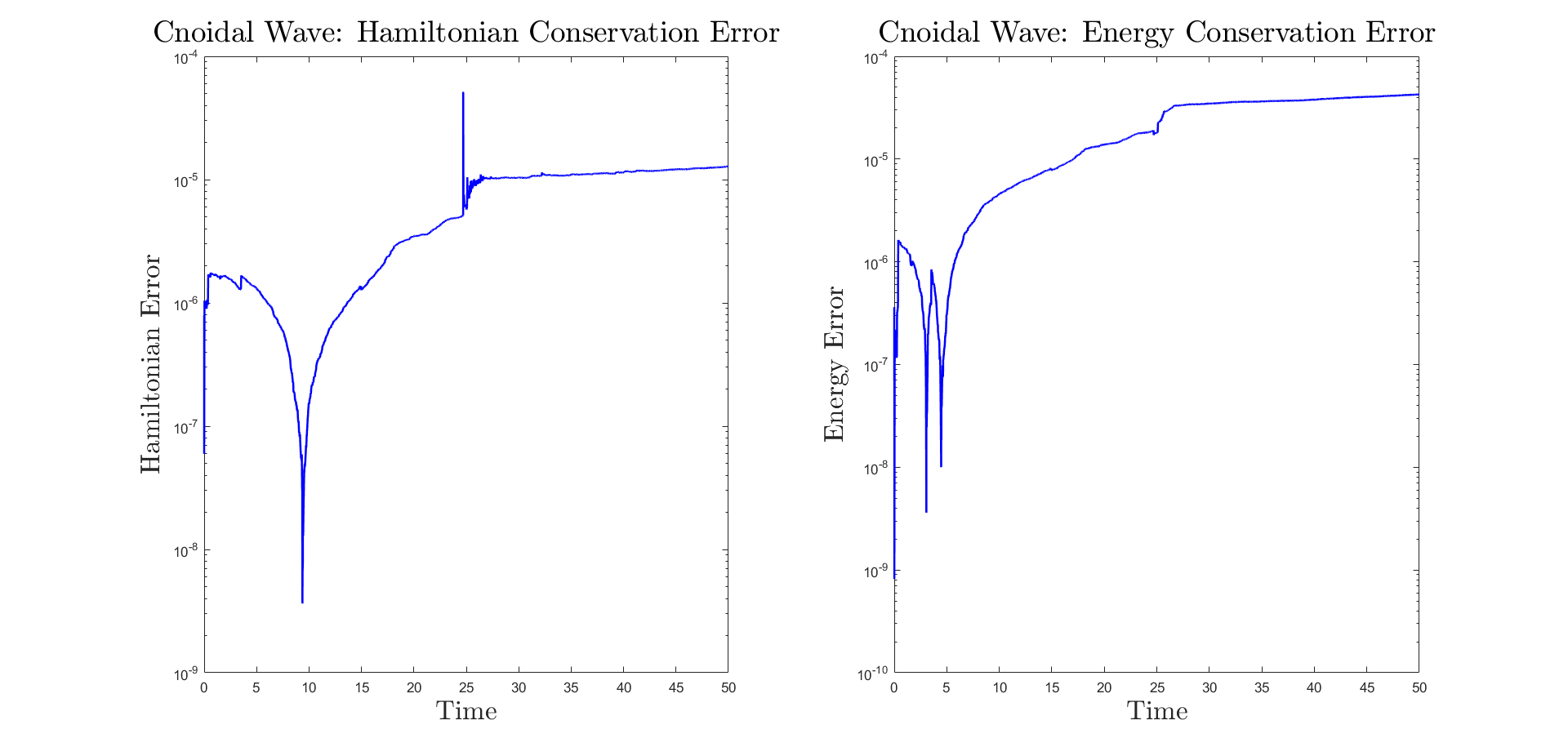}
\end{center}
\caption{Numerical Experiment 3 (classical KdV equation): Hamiltonian (Left) and energy (Right) conservation for the Cnoidal Wave. \dong{Shown on the second row are the corresponding errors.}}
\label{fig:graph9}
\end{figure}

\section{Concluding Remarks}
In this paper, we design and implement a \dong{new} conservative DG method for simulating solitary wave solutions to the generalized KdV equation. We prove that the method conserves the mass, energy and Hamiltonian of the solution. Numerical experiments confirm that our method does have the desirable conservation properties proved by our analysis. The convergence orders are also comparable to prior works by others. Future extensions include the investigation \dong{of} other choices of numerical fluxes, as well as applying the novel framework of devising \dong{new} conservative DG methods to other problems featuring physically interesting quantities that are conserved.

\appendix



\section{Implementation Details} \label{appendix:Implementation}
In the Appendix, we show how to rewrite the weak formulation of the DG method, \eqref{eq:scheme_time}, into the system \eqref{MatrixEqsF} for implementation using matrices and vectors. We start with the details on rewriting Eq. \eqref{eq:scheme_time1} into Eq. \eqref{eq:F1}. 
Assume that the interval $[-1, 1]$ is linearly mapped to the interval $I_i$ and the Legendre polynomial of degree $l$ on $[-1, 1]$ is correspondingly mapped to the polynomial $\phi_{i}^{l}(x)$ on the interval $I_{i}$ for $l = 0, \dong{\ldots},k$, and $i= 1, \dong{\ldots},N$. Then $u_h$ can be written as  
$u_{h}\rvert_{I_{i}} = \sum_{l=0}^{k} u_{i}^{l}(t)\phi_{i}^{l}(x)$, where $\{u^l_i (t)\}_{l=0}^k$ are degrees of freedom of $u_h$ on $I_i$  at time $t$.
Similar expansions are performed for $q_h$ and $p_h$.
Taking the test function $v=\phi_i^j(x)$ for $i=1,  \dong{\ldots}, N$ and $j=0, \dong{\ldots},k$ in \eqref{eq:scheme_time1} and using the definition of $\widehat{u_h}$, we get

%
\begin{align*}
 \textbf{M}[q]+\textbf{D}[u]+\textbf{A}[u] = 0,
\end{align*}
where $[u]=(u_1^0, \dong{\ldots},u_1^k, \dong{\ldots}, u_{N}^0, \dong{\ldots}, u_{N}^k)^T$ is the column vector that contains all the degrees of freedom of $u_h$, and $[q]$ and $[p]$ are the column vectors of degrees of freedom of $q_h$ and $p_h$, respectively. Here, the mass matrix \textbf{M} is block diagonal,
$$\textbf{M} =\dong{\textrm{diag}}(M^{I_{1}}, \dong{\ldots}, M^{I_{N}})$$ with components $$(M^{I_{i}})_{lj} = \int_{I_{i}} \phi_{i}^{l}(x) \phi_{i}^{j}(x) dx  $$
 for $j,l = 0,\dong{\ldots},k$ and $i = 1, \dong{\ldots}, N$. 
The stiffness matrix \textbf{D} is also block diagonal,
$$\textbf{D} =\dong{\textrm{diag}}(D^{I_{1}}, \dong{\ldots}, D^{I_{N}})$$  with components $$(D^{I_{i}})_{lj} = \int_{I_{i}} \phi_{i}^{l}(x) (\phi_{i}^{j})_{x}(x) dx$$ for $l,j = 0, \dong{\ldots}, k$ and $i = 1, \dong{\ldots}, N$. The matrix \textbf{A} is associated with the average flux in $\widehat{u_h}$. Note that a basis function on an interval $I_i$ only communicates with those on $I_i$ or on the two neighboring intervals $I_{i-1}$ and $I_{i+1}$. So the matrix \dong{\textbf{A}} is sparse and block diagonal. So are the matrices \textbf{D} and \textbf{M}. This is one of the advantages of DG methods which use local basis functions. 
Indeed, \textbf{A} is nearly block tridiagonal except the first and the last block rows. The three blocks used to assemble \textbf{A} have components as follows 
\begin{align*}
    (\textbf{A}_{I}^{-})_{jl} = \frac{(-1)^j}{2}, \quad  (\textbf{A}_{I}^{0})_{jl} = \frac{(-1)^{l+j}-1}{2}, \quad  (\textbf{A}_{I}^{+})_{jl} = -\frac{(-1)^l}{2} 
\end{align*}
for $j,l=0, \dong{\ldots},k$ and $i=1, \dong{\ldots}, N$.

Next, we rewrite the Eq. \eqref{eq:scheme_time2} into \eqref{eq:F2} in a similar way. The main difference lies in the term $\langle \dong{\widehat{q}_h},zn\rangle$. Using the definition of $\dong{\widehat{q}_h}$, we can rewrite this term as
\begin{align*}
    \langle\dong{\widehat{q}_h}, zn\rangle  &=  \langle\{q_{h}\}, z n \rangle  +  \tau_{qu}\langle[\![u_{h}]\!], z n \rangle.
\end{align*}
For the first term on the right hand side involving $\{q_h\}$, we can rewrite it as $\textbf{A} [q]$ using the average flux matrix \textbf{A}. For the second term that involves $\jmp{u_h}$, taking $z=\phi_i^l$, we have 
\begin{align*}
 \langle[\![u_{h}]\!], \phi_i^l n \rangle&= [\![u_{h}]\!](x_{i+1}) \phi_{i}^{l}(x_{i+1}^{-}) -  [\![u_{h}]\!](x_i)\phi_{i}^{l}(x_{i}^{+}) \\
 & = \bigg(\sum_{j=0}^{k} u_{i}^{j} - \sum_{j=0}^{k} u_{i+1}^{j}(-1)^{j}\bigg)  - \bigg(\sum_{j=0}^{k} u_{i-1}^{j}- \sum_{j=0}^{k} u_{i}^{j}(-1)^{j}\bigg) (-1)^{l}\\
 &=(-1)^{l+1}\sum_{j=0}^{k} u_{i-1}^{j}
  +(1+(-1)^{k+l})\sum_{j=0}^{k} u_{i}^{j}
  +(-1)^{j+1}\sum_{j=0}^{k} u_{i+1}^{j}
\end{align*}
for $ i = 1, \dong{\ldots}, N$, $l,j = 0, \dong{\ldots}, k$.
Note that for each $i$, the expression above only uses the interval $I_i$, the one before it, and the one after it. So we can write the term $\langle\jmp{u_h}, zn\rangle$ as
$$\langle\jmp{u_h}, zn\rangle=\textbf{J} [u],$$
 where \textbf{J} is a nearly block tridiagonal matrix except the first and last block rows. Now the Eq. \eqref{eq:scheme_time2} can be written as \eqref{eq:F2}.

To rewrite Eq. \eqref{eq:scheme_time3} as \eqref{eq:F3}, we just need to approximate the $(u_t, w)$ term by $\mathbf{M}([u]-[\bar{u}])/(\frac{1}{2}\Delta t)$. The rest terms are handled in a similar way to what we described above for \eqref{eq:scheme_time1} and \eqref{eq:scheme_time2}.

Rewriting the equations \eqref{eq:tau_pu1} and \eqref{eq:tau_qu1} that enforce the conservation of  Energy and Hamiltonian into \eqref{eq:F4} and \eqref{eq:F5} is \dong{straightforward}. We just need to use the notation $V_f = \sum_{i = 1}^{N}[\![V(u_{h})]\!] - \{\Pi f(u_{h})\}[\![u_{h}]\!](x_{i})$ and  $\eta_(\xi,\nu) = \sum_{i = 1}^{N}\jmp{\xi}\jmp{\nu}(x_i)$ for $\xi, \nu = u_h, q_h, $ or $p_h$.  

\bibliographystyle{siamplain}
\bibliography{Citationskdv}

\end{document}